\documentclass[10pt]{amsart}
\usepackage{mathrsfs}
\usepackage{amssymb, verbatim, amsmath,amsthm,amsfonts}
\usepackage{color}
\usepackage[all]{xy}
\usepackage{amsthm}
\usepackage{MnSymbol}
\usepackage{pifont}
\usepackage{enumitem}
\usepackage{cooltooltips}

\usepackage{graphicx} 
\usepackage[hidelinks]{hyperref}

\newcommand{\bc}{\begin{center}}
\newcommand{\ec}{\end{center}}
\newcommand{\be}{\begin{enumerate}}
\newcommand{\ee}{\end{enumerate}}
\newcommand{\beq}{\begin{equation}}
\newcommand{\eeq}{\end{equation}}
\newcommand{\bi}{\begin{itemize}}
\newcommand{\ei}{\end{itemize}}
\newcommand{\bd}{\begin{description}}
\newcommand{\ed}{\end{description}}
\newcommand{\ba}{\begin{array}}
\newcommand{\bea}{\begin{eqnarray*}}
\newcommand{\eea}{\end{eqnarray*}}
\newcommand{\ea}{\end{array}}
\newcommand{\bt}{\begin{tabular}}
\newcommand{\et}{\end{tabular}}

\newcommand{\mb}{\mbox}

\newcommand{\bmi}{\begin{minipage}}
\newcommand{\emi}{\end{minipage}}

\newcommand{\lb}{\linebreak}

\newtheorem{stel}{Theorem}[section]
\newtheorem{deflem}[stel]{Definition-Lemma}
\newtheorem{lemm}[stel]{Lemma}
\newtheorem{prop}[stel]{Proposition}

\newtheorem{exam}[stel]{Example}
\newtheorem{rem}[stel]{Remark}
\newtheorem{corollary}[stel]{Corollary}
\newtheorem{theo}[stel]{Theorem}

\newtheorem{defin}[stel]{Definition}
\usepackage{multicol}

\makeatletter
\newcommand{\myitem}[1]{%
\item[#1]\protected@edef\@currentlabel{#1}%
}
\makeatother

\makeatletter

\newcommand{\Matrix}[1]
    {\begin{pmatrix}
      \Matrix@r #1;\@bye;\Matrix@r
     \end{pmatrix}}

\def\Matrix@r #1;{\@bye #1\Matrix@z\@bye\Matrix@s #1,\@bye, }%
\def\Matrix@s #1,{#1\Matrix@t }%
\def\Matrix@t #1,{\@bye #1\Matrix@y\@bye\@firstofone {&#1}\Matrix@t}%
\def\Matrix@y #1\Matrix@t{\\ \Matrix@r }%
\def\Matrix@z #1\Matrix@r {}
\def\@bye  #1\@bye   {}
\makeatother

\title{\bfseries A note on the transport of (near-)field structures}
\author{Leandro Boonzaaier and Sophie Marques}

\begin{document}
\maketitle
\bc

\rm e-mail: leandro.boonzaaier@rain.co.za

\it\small
rain (Pty) Ltd, 
Cape Quarter,
Somerset Road, 
Green Point, 8005,\lb
South Africa\\

\rm e-mail: smarques@sun.ac.za

\it
Department of Mathematical Sciences, 
University of Stellenbosch, 
Stellenbosch, 7600,\lb
South Africa\\
\&
NITheCS (National Institute for Theoretical and Computational Sciences), 
South Africa \\ \bigskip

\ec

{\bf Abstract:} This paper addresses the question: given a scalar group, can we determine all the additions that transform this scalar group into a (near-)field? A key approach to addressing this problem involves transporting (near-)field structures via multiplicative automorphisms. We compute the set of continuous multiplicative automorphisms of the real and complex fields and analyze their structures. Additionally, we characterize the endo-bijections on the scalar group that define these additions.\\



\tableofcontents

\section*{Acknowledgement}
We express our sincere gratitude to Dr. Ben Blum-Smith for his invaluable contributions in providing us with the essential insights and arguments necessary for constructing the automorphisms of $\mathbb{R}$ and $\mathbb{C}$.

\section*{Introduction}\label{intro}
\noindent The term ``transport of structure" describes the process by which an object acquires a structure by being ``isomorphic" to another object that already possesses that structure (see \cite{Bourbaki}). This concept has been studied in various fields, including abstract algebra \cite{Arnal, Holm, Lusztig, Rentschler, Yaraneri};  calculus \cite{Buchmann}; algebraic geometry and number theory \cite{Bruyn, Hambleton}, \cite{Manasa}; and category theory \cite{Bennett, Diers}. This paper examines the transport of (near)-field structures within the framework of a fixed multiplicative monoid, motivated by connections to the theory of near-vector spaces.

\noindent  J.~André introduced the concept of a near-vector space in \cite{Andre} as a generalization of a vector space, where the distributive law holds only on one side. This allows for some non-linearity while still retaining many tools of linear algebra. These structures have been explored algebraically, geometrically, and categorically in recent literature (see, e.g., \cite{DeBruyn, Howell, Howell2, HM22, HR22, HS18, MMJ}). In this paper, a system \((F, \cdot, 1, -1, 0)\) is called a scalar group if:  
\begin{itemize}
\item \((F, \cdot, 1)\) forms a monoid;  
\item \((F \setminus \{0\}, \cdot, 1)\) forms a group;  
\item for all \(\alpha \in F\), we have \(0 \cdot \alpha = \alpha \cdot 0 = 0\);  
\item \(\pm 1\) are solutions to \(x^2 - 1 = 0\).  
\end{itemize} 
A near-vector space on a scalar group \((F, \cdot, 1, -1, 0)\) defines different additions that turn \(F\) into a near-field. Specifically, each non-zero element \(u\) of the quasi-kernel induces an addition \(+_u\) such that \((F, +_u, \cdot)\) forms a near-field (see \cite{Andre, DeBruyn, MMJ}).  

\noindent In this context, it is natural to pose the following question:  

\begin{center}  
  {\sf Given a scalar group, how can we determine all the additions that transform this scalar group into a near-field?}  
\end{center}  

\noindent Studying this question reveals new structures, transport of structures, and relationships that enrich the fields of algebra while paving the way for further applications and generalizations. A complete answer for a given scalar group would enable us to construct all near-vector spaces over that scalar group, combining  \cite[Theorem 2.4-17]{DeBruyn} and \cite[Theorem 2.5-2]{DeBruyn}. Since near-vector spaces generalize vector spaces while preserving many of their structural and algebraic properties, addressing this question would significantly broaden the scope of applications for near-linear algebra.

\noindent The implications of this question go beyond the theory of near-vector spaces. Classifying additions that transform a scalar group into a near-field enhances our understanding of the algebraic diversity and properties of near-fields, and more broadly, near-ring structures. These have been extensively studied in foundational works such as \cite{ Clay, Ferrero, Pilz, Pilz2, Wähling, Zassenhaus}.  The problem also has strong connections to number theory. For instance, the paper concludes by showing that if an addition \(\boxplus\) on \(\mathbb{Q}\) exists such that \((\mathbb{Q}, \boxplus, \cdot)\) forms a field isomorphic to another field \((K, +, \cdot)\), then \((\mathbb{Q}, +, \cdot)\) and \((K, +, \cdot)\) share fundamental arithmetic properties. Resolving this question would thus reveal significant number-theoretic relationships between such fields. Furthermore, it is shown that there is a bijection between two near-field structures sharing the same multiplication if and only if this bijection is a multiplicative automorphism of \((F, \cdot)\) (see Proposition \ref{bij+}).

\noindent A significant portion of this paper involves computing these multiplicative automorphisms for different fields: finite fields, the real field, and the complex field (see \(\S3\)). The real and finite fields are relatively straightforward to compute, while the multiplicative automorphisms of the complex field exhibit a rich structure and an intriguing representation in the usual complex plane (see Appendix). Other authors have studied various types of automorphisms of the complex field and their properties (see \cite{conradauto, Kest, Salz, Sound, Yale}).  

\noindent The first section of this paper introduces the concept of transport of structure, a natural method for inducing new additions on a scalar group. In the second section, we compute and analyze the structure of the set of multiplicative automorphisms for specific fields. The key result in this section is the computation and characterization of the multiplicative automorphisms of the complex field (see Theorem \ref{autoC} and Proposition \ref{groupaut}). Lastly, the final section explores the definition of additions on a scalar group that are not derived from the transport of structure from a given field. These additions are characterized as endo-bijections on \((F, \cdot)\) satisfying specific axioms (see Definition \ref{nfam} and Definition-Lemma \ref{nfa}).

\newpage
\tableofcontents

\newpage

\section{Transport of ring structures}\label{induced_fields}

In the context of near-ring theory, we introduce the following definitions to clarify the underlying group structures and the different automorphisms associated with a near-ring. 

\begin{defin}
Let \((R, +, \cdot)\) be a near-ring.
\begin{enumerate}
\item We say that \((R, +)\) is the \textsf{underlying additive group of the near-ring \((R, +, \cdot)\)}.
\item We say that \(\phi\) is an \textsf{additive automorphism of \(R\)} if \(\phi\) is an automorphism of the additive group \((R, +)\).
\item We say that \((R, \cdot)\) is the \textsf{underlying multiplicative monoid of the near-ring \((R, +, \cdot)\)}.
\item We say that \(\phi\) is a \textsf{multiplicative automorphism of \(R\)} if \(\phi\) is an automorphism of the monoid \((R, \cdot)\).
\end{enumerate}
\end{defin}
For any field \((K,+ , \cdot)\), it is worth noting that we can extend an automorphism \(\phi\) of \((K^*, \cdot) \) to an automorphism of \((K, \cdot)\) simply by mapping \(0\) to \(0\). However, when \(\phi\) is continuous, the continuity of \(\phi\) is not automatically extended to \(0\).

Given a set $F$, we say that a bijection from $F$ to $F$ is an {\sf endo-bijection on $F$}.
Consider two monoids with zero elements \((M_1, \cdot_1, 1_1, 0_1)\) and \((M_2, \cdot_2, 1_2, 0_2)\), and let \(\phi\) be a multiplicative morphism from \((M_1, \cdot_1, 1_1, 0_1)\) to \((M_2, \cdot_2, 1_2, 0_2)\). We have the following properties:

\begin{itemize}
\item If \(0_2 \in \operatorname{Range} (\phi)\), then \(\phi(0_1) = 0_2\).
\item For any invertible element \(\alpha \in M_1\), it holds that \(\phi(\alpha^{-1}) = \phi(\alpha)^{-1}\).
\end{itemize}

In particular, when \((F_1, +_1, \cdot_1)\) and \((F_2, +_2, \cdot_2)\) are fields, any multiplicative bijection \(\varphi\) from \((F_1, +_1, \cdot_1)\) to \((F_2, +_2, \cdot_2)\) maps a primitive \(n\)th root of unity in \(F_1\) to a primitive \(n\)th root of unity in \(F_2\). Specifically, for \(F_1 = F_2 = \mathbb{C}\), \(\phi(i) \in \{ \pm i \}\).

 We now recall the concept of transport of structure via a bijection and outline some fundamental properties. The proof is omitted as it is straightforward and can be readily verified.

\begin{deflem}\label{indfield}
Let \((R, +, \cdot)\) be a near-ring, \(S\) be a set, and \(\phi\) be a bijective map from \(S\) to \(R\).
\begin{enumerate}
\item We define \textsf{the near-ring transported by \(\phi\)}, denoted by \((S, +_\phi, \cdot_\phi)\) or simply \(S_\phi\), as the near-ring whose transported multiplication \(\cdot_\phi\) is given by:
\[
a \cdot_\phi b = \phi^{-1} (\phi(a) \cdot \phi(b))
\]
and transported addition \(+_\phi\) is given by:
\[
a +_\phi b = \phi^{-1} (\phi(a) + \phi(b))
\]
for any \(a, b \in S\). For any \(a \in S\), we denote 
\begin{itemize} 
\item \(\phi^{-1}(a)\) by \(a_\phi\). Note that \(0_\phi\) is the identity element with respect to \(+_\phi\),
\item \(\phi^{-1}(-\phi(a))\) by \(-_\phi a\). Note that \(-_\phi a\) is the additive inverse of \(a\) in \(S\) with respect to \(+_\phi\).
\end{itemize}
When \((R, +, \cdot)\) is a ring, near-field, field, or division ring, so is \((S, +_\phi, \cdot_\phi)\). Moreover, \(\phi\) is a near-ring isomorphism between \(S_\phi\) and \(R\).
\item When \(S=R\) and \(\phi\) is a multiplicative automorphism of \(R\), we obtain \(\cdot_\phi = \cdot\). Also, \(0\) is the identity element with respect to \(+_\phi\), and \(-a\) is the additive inverse of \(a\) in \(R\) with respect to \(+_\phi\) for any \(a \in R\). In other words, we keep the multiplication but change only the addition to create a new ring structure on the ring.
\item When \(S=R\) and \(\phi\) is an additive automorphism of \(R\), we obtain \(+_\phi = +\). In other words, we keep the addition but change only the multiplication to create a new ring structure on the ring.
\end{enumerate}
\end{deflem}

Here is a simple example to illustrate the above.
\begin{exam}
Let \((\mathbb{Q}(\alpha), + , \cdot)\) be a number field where \(\alpha\) is a primitive element over \(\mathbb{Q}\) of degree \(n\). There exists a bijection \(\phi\) from any number field \(\mathbb{Q}(\alpha)\) to \(\mathbb{Q}\). For instance, we can consider a bijection from \(\mathbb{Q}\) to \(\mathbb{N}\) and a bijection from \(\mathbb{N}^n\) to \(\mathbb{N}\) recursively from a Cantor pairing from \(\mathbb{N}^2\) to \(\mathbb{N}\) in order to construct a bijection \(\phi\) from \(\mathbb{Q}(\alpha)\) to \(\mathbb{Q}\). Then we endow \(\mathbb{Q}\) with the operations \(x +_\phi y= \phi^{-1} (\phi(x) + \phi(y))\) and \(x \cdot_\phi y =\phi^{-1} (\phi(x) \cdot \phi(y))\) where \(x, y \in \mathbb{Q}\). Thus, \((\mathbb{Q}, +_\phi, \cdot_\phi)\) is isomorphic to \((\mathbb{Q}(\alpha), + , \cdot)\).
\end{exam}

We now prove that if any near-field structure transported by a endo-bijection that preserve the multiplication, this endo-bijection is a multiplicative automorphism. 

 \begin{prop} \label{bij+}
  Let $(F, + , \cdot)$ be a near-field of characteristic $p$ and $\phi$ be an endo-bijection on $F$. The following assertions are equivalent: 
\begin{enumerate} 
\item $(F, +_\phi, \cdot)$ is a near-field;
\item $\alpha (\beta +_\phi \gamma) =   \alpha \beta +_\phi \alpha \gamma
$, for any $\alpha, \beta, \gamma \in F$. 
\item $\phi$ is a multiplicative automorphism of $F$. 
\end{enumerate}
 \end{prop} 
 \begin{proof} The equivalence \((1) \Leftrightarrow (2)\) is clear. To establish the equivalence \((2) \Leftrightarrow (3)\), applying \((2)\) to \(\alpha = 0\), we obtain \(0 = 0 +_\phi 0\). That is, \(\phi(0) = \phi(0) + \phi(0)\). From this, we deduce \(\phi(0) = 0\).

Now, let \(\alpha \in F^*\) and \(\beta \in F\). Since we have \(\alpha (\alpha^{-1} +_\phi \beta) = 1 +_\phi \alpha \beta\), we obtain:
\[ \phi(\alpha\beta) = \phi(\alpha) \phi(\beta) + \phi(\alpha) \phi(\alpha^{-1}) - 1. \]
Applying this equality to \(\beta = 0\), we obtain:
\[ 0 = \phi(\alpha) \phi(\alpha^{-1}) - 1. \]
Thus, we get:
\[ \phi(\alpha\beta) = \phi(\alpha) \phi(\beta). \]

Therefore, we deduce the implication \((2) \Rightarrow (3)\). The converse is not hard to established.
\end{proof}

We introduce a summation notation for the transported sum.

\begin{defin}
Let \((R, +, \cdot)\) be a near-ring, \(S\) be a set, and \(\phi\) be a bijective map from \(S\) to \(R\). 
We denote
\[{}^\phi \! \sum_{i=1}^n a_i = a_1 +_\phi a_2 +_\phi \cdots +_\phi a_{n-1} +_\phi a_n \]
and
\[{}^\phi \! \prod_{i=1}^n a_i = a_1 \cdot_\phi a_2 \cdot_\phi \dots \cdot_\phi a_{n-1} \cdot_\phi a_n,\]
where \(n \in \mathbb{N}\) and \(a_i \in S\) for all \(i \in \{1, \ldots, n\}\).

In particular, we have
\[{}^\phi \sum_{i=1}^n a = n_\phi \cdot a \quad \text{and} \quad {}^\phi \! \prod_{i=1}^n a = \phi^{-1}(\phi(a)^n).\]
\end{defin}

We note that for any \(n \in \mathbb{N}\) such that \((n, \operatorname{Char}(K))=1\) and \(a \in S\), \(n_\phi^{-1} \cdot a\) is the element of \(S\) such that \({}^\phi \sum_{i=1}^n n_\phi^{-1} \cdot a = a\).

The following lemma precisely establishes the conditions under which two transported structures become equal.

\begin{lemm} \label{equal}
Let \((R, +, \cdot)\) be a near-ring, and let \(\phi\) and \(\psi\) be bijective maps from \(S\) to \(R\). Then \(S_\phi = S_\psi\) if and only if \(\psi \circ \phi^{-1}\) is a near-ring automorphism of \((R, +, \cdot)\).
\end{lemm}

\begin{proof}
The equality \(S_\phi = S_\psi\) means that for all \(a, b \in S\), we have \(a +_\phi b = a +_\psi b\) and \(a \cdot_\phi b = a \cdot_\psi b\). This is equivalent to saying that \(\phi^{-1}(\phi(a) + \phi(b)) = \psi^{-1}(\psi(a) + \psi(b))\) and \(\phi^{-1}(\phi(a) \cdot \phi(b)) = \psi^{-1}(\psi(a) \cdot \psi(b))\) for all \(a, b \in S\).

Since \(\phi\) is a bijection, for all \(a, b \in S\), there exist \(a', b' \in R\) such that \(a = \phi^{-1}(a')\) and \(b = \phi^{-1}(b')\). Substituting these into the previous equalities, we get:
\[
a' + b' = (\psi \circ \phi^{-1})^{-1}(\psi(\phi^{-1}(a')) + \psi(\phi^{-1}(b')))
\]
and
\[
a' \cdot b' = (\psi \circ \phi^{-1})^{-1}(\psi(\phi^{-1}(a')) \cdot \psi(\phi^{-1}(b'))).
\]

This shows that \(\psi \circ \phi^{-1}\) preserves both addition and multiplication, thus proving that \(\psi \circ \phi^{-1}\) is a near-ring automorphism of \((R, +, \cdot)\).
\end{proof}

Building upon the preceding lemma, we can readily derive the following corollary.

\begin{corollary} \label{cor1}
Let \((R, +, \cdot)\) be a near-ring. We have
\[
\mathcal{R}(R) \simeq \operatorname{Bij}(R) / \operatorname{Aut}(R, +, \cdot), \quad \mathcal{R}_{\operatorname{m}}(R) \simeq \operatorname{Aut}(R, \cdot) / \operatorname{Aut}(R, +, \cdot),
\]
and
\[
\mathcal{R}_{\operatorname{a}}(R) \simeq \operatorname{Aut}(R, +) / \operatorname{Aut}(R, +, \cdot),
\]
where \(\mathcal{R}(R)\) (resp. \(\mathcal{R}_{\operatorname{m}}(R)\), resp. \(\mathcal{R}_{\operatorname{a}}(R)\)) is the set of rings transported by endo-bijections (resp. multiplicative, resp. additive automorphisms) of \(R\).
\end{corollary}

\section{The multiplicative automorphisms of certain fields}
In this section, our aim is to compute the set of multiplicative automorphisms and study the structure of these sets for specific fields, with a focus on the continuous automorphisms of the field of complex numbers. 

\subsection{Review of some basic results}
We start with the straightforward case of a finite field. Given a prime number \(p\) and \(n \in \mathbb{N}\), for a finite field \(K\) of order \(p^n\), it is well known that the multiplicative group of such a field is cyclic of order \(p^n-1\) (as referenced in \cite[Theorem 1.1]{conradcyclotomic} and \cite[Proposition 4.3]{topicsinfieldtheory}). We establish that the automorphisms of \((K, \cdot)\) are isomorphic to the unit group of \(\mathbb{Z}/(p^n-1)\mathbb{Z}\). Simply put, for any \(\psi \in \operatorname{Aut}(K, \cdot)\), there exists \(\alpha \in \mathbb{N}\) such that \(\psi(x) = x^\alpha\) for all \(x \in K\), where \(\alpha\) is invertible modulo \(p^n-1\). Additionally, we recall that \(\operatorname{Aut}(K, +, \cdot) = \langle F_p \rangle\), where \(F_p: K \rightarrow K\) represents the Frobenius automorphism sending \(x\) to \(x^p\), and \(\langle F_p \rangle\) denotes the group generated by \(F_p\) in \((U_{p^n-1}, \odot)\). Therefore, from Corollary \ref{cor1}, we deduce that:
\[
\mathcal{M}(K)_{\operatorname{aut}} \simeq U_{p^n-1}/ \langle [p] \rangle
\]
where \(\mathcal{M}(K)_{\operatorname{aut}}\) is the set of all fields induced by multiplicative automorphisms of \(K\).

To compute the automorphisms of the real or complex numbers, we note that for any field \(K\), the set of automorphisms of the additive group \((K, +)\) is the set of automorphisms of \(K\) when viewed as a vector space over its prime subfield. Indeed, an automorphism of \((K, +)\) is also an automorphism of \(K\) considered as a \(\mathbb{Z}\)-module. Moreover, this automorphism naturally becomes a \(\mathbb{Z}/p\mathbb{Z}\)-vector space morphism when the characteristic of \(K\) is \(p\), and a \(\mathbb{Q}\)-vector space morphism when the characteristic of \(K\) is \(0\). If we further assume that these additive automorphisms are continuous, we can utilize the fact that \(\mathbb{Q}\) is a dense subset of \(\mathbb{R}\) to obtain that the set of the continuous additive automorphism of $R$ is given by:
\[
\operatorname{Aut}_{\operatorname{cont}}(\mathbb{R}, +) = \operatorname{Aut}_{\mathbb{R}}(\mathbb{R}) = \{\psi_\alpha \mid \alpha \in \mathbb{R}^*, \psi_\alpha(x) = \alpha x, \forall x \in \mathbb{R} \} \simeq (\mathbb{R}^*, \cdot).
\]
We note that \(\psi_\alpha^{-1} = \psi_{\frac{1}{\alpha}}\). Therefore, the ring structures transported by the additive morphisms of \(\mathbb{R}\) are \((\mathbb{R}, +, \cdot_\alpha)\) where \(\alpha \in \mathbb{R}^*\) and \(\cdot_\alpha\) is given by \(x \cdot_\alpha y = \alpha \cdot x \cdot y\), for all \(x, y \in \mathbb{R}\).

Moreover, using the fact that \(\{1, i\}\) forms a basis for \(\mathbb{C}\) as an \(\mathbb{R}\)-vector space, we have:
\[
\operatorname{Aut}_{\operatorname{cont}}(\mathbb{C}, +) = \operatorname{Aut}_{\mathbb{R}}(\mathbb{C}) = \{\psi_{\alpha, \beta} \mid \alpha \in \mathbb{C}, \beta \in \mathbb{C} \setminus \alpha \mathbb{R} \} \simeq (GL(2, \mathbb{R}), \cdot)
\]
where \(\psi_{\alpha, \beta}\) is the morphism sending \(x + iy\) to \(\alpha x + \beta y\), for all \(x, y \in \mathbb{R}\) with \(\alpha \in \mathbb{C}\) and \(\beta \in \mathbb{C} \setminus \alpha \mathbb{R}\). Writing \(\alpha = a_1 + i a_2\) and \(\beta = b_1 + i b_2\), we denote by \({\sf det}(\alpha, \beta) = a_1 b_2 - a_2 b_1\), we note that by assumption, \({\sf det}(\alpha, \beta)\neq 0\) and $$\psi_{\alpha, \beta}^{-1} = \psi_{{\sf det}(\alpha, \beta)^{-1} (b_2 - a_2 i), {\sf det}(\alpha, \beta)^{-1}(-b_1 + i a_1)}.$$ Therefore, the ring structures transported by the additive morphisms of \(\mathbb{C}\) are \((\mathbb{C}, +, \cdot_{\alpha, \beta})\) where \(\alpha \in \mathbb{C}\) and \(\beta \in \mathbb{C} \setminus \alpha \mathbb{R}\) and \(\cdot_{\alpha, \beta}\) is given by 
$$(x + iy) \cdot_{\alpha, \beta} (x' + iy') = \psi_{\alpha, \beta}^{-1}((\alpha x + \beta y)(\alpha x' + \beta y')),$$ for all \(x + iy, x' + iy' \in \mathbb{C}\). 

The action of $\psi_{\alpha, \beta}$ on the complex plane is a transformation of the coordinate system as illustrated in the figure below.
\begin{multicols}{2}
\begin{center}
\includegraphics[scale=0.4]{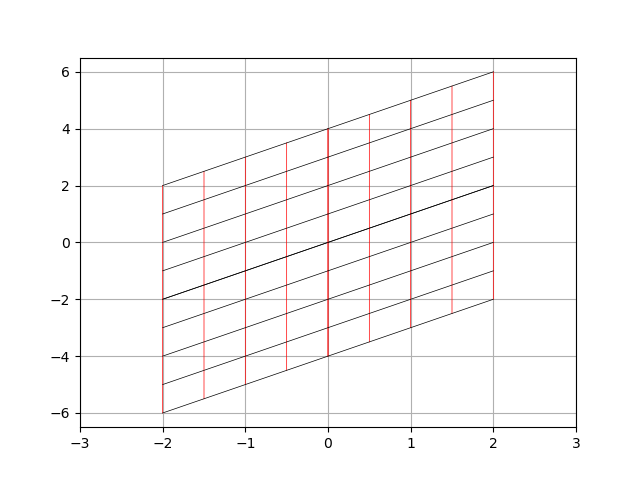} 
{$\alpha=1+i$ and $\beta=2i$}
\end{center}
\columnbreak
\begin{center}
\includegraphics[scale=0.4]{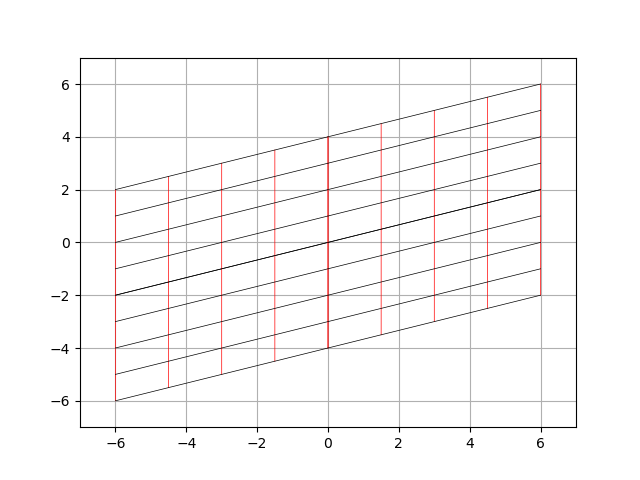}
{$\alpha=3+i$ and $\beta=2i$} 
\end{center}
\end{multicols}

A direct consequence of the above results is the well-established fact that automorphisms of the field of real numbers are trivial, whereas in the complex field, they are either the complex conjugation or the identity morphism.

We now briefly recall how one can compute the set of multiplicative automorphisms of the real field before turning our full attention to those of the complex field. Utilizing the characterization of the automorphisms of \((\mathbb{R}, +)\), we can straightforwardly derive the multiplicative automorphisms of \(\mathbb{R}_{>0}\) by simply using the exponential map onto \(\mathbb{R}_{>0}\). More precisely, the continuous automorphisms of \((\mathbb{R}_{>0}, \cdot)\) are precisely the maps \(\phi_\alpha: \mathbb{R}_{>0} \rightarrow \mathbb{R}_{>0}\) sending \(x\) to \(x^\alpha\), for some \(\alpha \in \mathbb{R}^*\).

There exists a direct correspondence between the multiplicative automorphisms of \(\mathbb{R}\) and the automorphisms of \((\mathbb{R}_{>0}, \cdot)\), achieved through the restriction and corestriction of the multiplicative automorphisms of \(\mathbb{R}\) onto \(\mathbb{R}_{>0}\). This follows from the fact that every positive real number is a square, and squares are mapped to squares by multiplicative automorphisms. Additionally, for any \(x > 0\), it follows that \(\psi(-x) = -\psi(x)\), and \(\psi(0) = 0\) for any multiplicative automorphism \(\psi\) of \(\mathbb{R}\).

More precisely, the multiplicative automorphisms of \(\mathbb{R}\) are described as maps of the following form:
\[
\psi(x) = 
\begin{cases} 
0 & \text{if } x=0\\
\eta(x) & \text{if } x > 0 \\
-\eta(-x) & \text{if } x < 0
\end{cases}
\]
where \(\eta\) is an automorphism of \((\mathbb{R}_{>0}, \cdot)\).

Any automorphism \(\eta\) of \((\mathbb{R}_{>0}, \cdot)\) is of the form \(\exp \circ \chi \circ \ln\), where \(\chi\) is an automorphism of \((\mathbb{R}, +)\). Since \(\chi\) is an automorphism of \(\mathbb{R}\) as a \(\mathbb{Q}\)-vector space, by basic linear algebra, it is uniquely defined by choosing two Hamel bases of \(\mathbb{R}\) over \(\mathbb{Q}\), say \(\{ b_i \}_{i \in I}\) (resp. \(\{ c_i \}_{i \in I}\)), and setting \(\chi(b_i) = c_i\) for all \(i \in I\). (Here, by Hamel basis for \(\mathbb{R}\) over \(\mathbb{Q}\), we mean a basis of \(\mathbb{R}\) as a vector space over \(\mathbb{Q}\), see also \cite[p. 69]{Golan}).

Finally, the multiplicative automorphisms of $\mathbb{R}$ continuous over \(\mathbb{R}^*\) (resp. \(\mathbb{R}\)) are precisely given by 
\[
\epsilon_\alpha(x) = 
\begin{cases}
x^\alpha & \text{if } x \geq 0 \ \\
-(-x)^\alpha & \text{if } x < 0
\end{cases}
\]
for \(\alpha \in \mathbb{R}^*\) (resp. \(\alpha \in \mathbb{R}^*_{>0}\)). 
Moreover, the inverse of \(\epsilon_\alpha\), denoted as \(\epsilon_\alpha^{-1}\), equals \(\epsilon_{1/\alpha}\) and \(\epsilon_\alpha\) is a homeomorphism. We note that a multiplicative automorphism \(\epsilon_\alpha\) of \(\mathbb{R}\) induces an automorphism of \(\mathbb{Q}\) if and only if \(\alpha \in  \{\pm 1\}\).

Using the notations as above, we summarize the result established with the following group isomorphisms:
\[
(\mathbb{R}^*, \cdot) \simeq^{\Psi_1} (\operatorname{Aut}_{\operatorname{cont}} (\mathbb{R}^*, \cdot), \circ) \simeq^{\Psi_2} (\operatorname{Aut}_{\operatorname{cont}} (\mathbb{R}_{>0}, \cdot), \circ) \simeq^{\Psi_3} (\operatorname{Aut}_{\operatorname{cont}} (\mathbb{R}, +), \circ)
\]
where \(\Psi_1 (\alpha) = \epsilon_\alpha\), \(\Psi_2 (\epsilon_\alpha) = \phi_\alpha\), and \(\Psi_3 (\phi_\alpha) = \psi_\alpha\), for any \(\alpha \in \mathbb{R}^*\). 
Moreover, $\Psi_1$ induces the bijection, 
\[
(\mathbb{R}_{>0}, \cdot) \simeq^{\Psi_1} (\operatorname{Aut}_{\operatorname{cont}} (\mathbb{R}, \cdot), \circ). \]
We observe that these isomorphisms allow us to transfer the topology of \(\mathbb{R}^*\) to the subsequent groups of automorphisms.

We end this section by observing that we can also easily describe the multiplicative automorphisms of the rings \(\mathbb{Z}\) and \(\mathbb{Q}\) as follows:
\begin{enumerate}
\item For \(\mathbb{Z}\):
Multiplicative automorphisms are precisely obtained by multiplicatively extending a bijection between the set of prime natural numbers and a set of the form  
\(\{\eta_p p \mid p \text{ is a prime natural number}\}\), where \(\eta_p \in \{\pm 1\}\), for all \(p\) prime natural number and mapping \(0\) to \(0\) and \(\pm 1\) to \(\pm 1\).

\item For \(\mathbb{Q}\):
Multiplicative automorphisms are precisely obtained by multiplicatively extending a bijection between the set of prime natural numbers and a set of the form 
\(\{\eta_p p^{\nu_p} \mid p \text{ is a prime natural number}\}\), where \(\eta_p, \nu_p \in \{\pm 1\}\) for all \(p\) prime natural number, while mapping \(0\) to \(0\) and \(\pm 1\) to \(\pm 1\).
\end{enumerate}

We conclude this section with an example that demonstrates that \(\mathbb{Q}\) can be endowed with an addition, \(+_\phi\), such that \((\mathbb{Q}, +_\phi, \cdot)\) is not isomorphic to \((\mathbb{Q}, +, \cdot)\). Let \(K\) be a number field such that its ring of integers is a PID with a group of units \(\{\pm 1\}\). For instance, considering the field \(K = \mathbb{Q}(\sqrt{-19})\) with ring of integer \(\mathcal{O}_K = \mathbb{Z}\left[\frac{1+\sqrt{-19}}{2}\right]\). Since \(\mathcal{O}_K\) is a principal ideal domain, \(\mathcal{O}_K\) is multiplicatively generated by its irreducible elements (those are also prime elements). Since above any prime number in \(\mathbb{Z}\), there is a finite number of prime elements of \(\mathcal{O}_K\), the cardinality of the irreducible elements of \(\mathcal{O}_K\) is the same as \(\mathbb{N}\). We can therefore create a multiplicative bijection \(\varphi\) from \(K\) to \(\mathbb{Q}\), extending by multiplicativity a bijection from a complete set of irreducible elements of \(\mathbb{Z}\left[\frac{1+\sqrt{-19}}{2}\right]\), distinct up to units, into a complete set of prime numbers of \(\mathbb{Z}\), distinct up to unit, sending \(\pm 1\) to \(\pm 1\). Then \((\mathbb{Q}, +_\phi, \cdot) \simeq (\mathbb{Q}(\sqrt{-19}), +, \cdot)\). We note that we cannot do such a construction with \(\mathbb{Q}(i)\) since the group of units has order 4. Indeed, a primitive 4th root of unity must map to a primitive 4th root of unity via a multiplicative isomorphism, but there is no primitive 4th root of unity in \(\mathbb{Q}\).
\subsection{Multiplicative automorphisms of $\mathbb{C}$}
\subsubsection{From $\operatorname{Aut}_{\operatorname{cont}} (\mathbb{C}, + )^{{2\pi i \mathbb{Z}} } $ to $\operatorname{Aut}_{\operatorname{cont}} ( \mathbb{C},\cdot)$}
The lemma presented below allows us to characterize the continuous multiplicative automorphisms of $\mathbb{C}^*$ in relation to the continuous automorphisms of $(\mathbb{C}, +)$. The proof uses algebraic topology to obtain the result wanted. 
\begin{lemm} \label{gpautoC}
We have the following group isomorphisms:
 $$  (\operatorname{Aut}_{\operatorname{cont}} (\mathbb{C}, + )^{{2\pi i \mathbb{Z}} }, \circ ) \simeq^{\Psi} \left( \operatorname{Aut}_{\operatorname{cont}}\left( \frac{\mathbb{C}}{2\pi i \mathbb{Z}}  , + \right), \circ \right) \simeq^{\Theta} (\operatorname{Aut}_{\operatorname{cont}} ( \mathbb{C}^*,\cdot),\circ )   $$
where
\begin{itemize}
\item $\operatorname{Aut}_{\operatorname{cont}} (\mathbb{C} , +)^{{2\pi i \mathbb{Z}} }$ is the group of continuous automorphisms of $(\mathbb{C}, +)$ that fix $2\pi i \mathbb{Z}$ set-wise.
\item $\operatorname{p} : \mathbb{C} \rightarrow \frac{\mathbb{C}}{2\pi i \mathbb{Z}}$ is the quotient map.
\item $\widetilde{\operatorname{exp}} : \frac{\mathbb{C}}{2\pi i \mathbb{Z}} \rightarrow \mathbb{C}^*$ is the canonical isomorphism induced by the complex exponential $\operatorname{exp} : \mathbb{C} \rightarrow \mathbb{C}^*$, via the first isomorphism theorem.
\item $L$ is the branch of $\operatorname{Log}$ defined by $L(z) = \operatorname{ln}(r) + i \operatorname{arg}(s)$, with $\operatorname{ln}$ denoting the natural logarithm on $\mathbb{R}_{> 0}$ and $\operatorname{arg}$ is the map that sends a complex number to its principal argument. 
\item $\Psi : \operatorname{Aut}_{\operatorname{cont}} (\mathbb{C} , + )^{{2\pi i \mathbb{Z}} } \rightarrow \operatorname{Aut}_{\operatorname{cont}}\left(\frac{\mathbb{C}}{2\pi i \mathbb{Z}}, +\right)$ sends $\phi$ to the automorphism $\widetilde{\phi} $ of $\left(\frac{\mathbb{C}}{2\pi i \mathbb{Z}}, +\right)$ uniquely determined by the equality $\widetilde{\phi} \circ \operatorname{p} = \operatorname{p} \circ \phi.$
\item $\Theta : \operatorname{Aut}_{\operatorname{cont}}(\frac{\mathbb{C}}{2\pi i \mathbb{Z}}, +) \rightarrow \operatorname{Aut}_{\operatorname{cont}}(\mathbb{C}^*,\cdot)$ sends $\phi$ to $\widetilde{\operatorname{exp}} \circ \phi \circ \widetilde{\operatorname{exp}}^{-1}$.
 \end{itemize} 
In particular,  $\Theta \circ \Psi : \operatorname{Aut}_{\operatorname{cont}} (\mathbb{C} , +)^{{2\pi i \mathbb{Z}} } \rightarrow (\operatorname{Aut}_{\operatorname{cont}} ( \mathbb{C}^*,\cdot),\circ ) $ sends $\phi$  to $\operatorname{exp} \circ \, \phi  \circ \operatorname{L}$. 
\end{lemm} 
\begin{proof}
Let us consider the surjective homomorphism $\operatorname{exp}: \mathbb{C} \rightarrow \mathbb{C}^*$ with the kernel $2\pi i \mathbb{Z}$. Via the first isomorphism theorem, $\operatorname{exp}$ induces an isomorphism $\widetilde{\operatorname{exp}}: \frac{\mathbb{C}}{2\pi i \mathbb{Z}} \rightarrow \mathbb{C}^*$. From this, we can deduce easily that $\Theta$ is a bijection from $\operatorname{Aut}_{\operatorname{cont}}(\frac{\mathbb{C}}{2\pi i \mathbb{Z}}, +)$ to $\operatorname{Aut}_{\operatorname{cont}}(\mathbb{C}^*, \cdot)$. Next, we prove that $\operatorname{Aut}_{\operatorname{cont}}(\frac{\mathbb{C}}{2\pi i \mathbb{Z}}, +) \simeq \operatorname{Aut}_{\operatorname{cont}}(\mathbb{C}, +)^{{2\pi i \mathbb{Z}}}$. We will begin by proving that any continuous automorphism $\phi$ of $\mathbb{C}$ that fixes $2\pi i \mathbb{Z}$ as a set will induce a continuous automorphism of the quotient group $\frac{\mathbb{C}}{2\pi i \mathbb{Z}}$.

 Since $\operatorname{p}$ is a continuous surjective morphism, the composition $\operatorname{p} \circ \phi: \mathbb{C} \rightarrow \frac{\mathbb{C}}{2\pi i \mathbb{Z}}$ is a surjective continuous homomorphism. By the first isomorphic theorem, this map induces a unique continuous automorphism $\widetilde{\phi}: \frac{\mathbb{C}}{2\pi i \mathbb{Z}} \rightarrow \frac{\mathbb{C}}{2\pi i \mathbb{Z}}$, such that $\widetilde{\phi} \circ \operatorname{p} = \operatorname{p} \circ \phi$, due to the fact that $\operatorname{p} \circ \phi$ has $2\pi i \mathbb{Z}$ as its kernel.

Therefore, we define a homomorphism 
$$\begin{array}{cccc}  \Psi : & \operatorname{Aut}_{\operatorname{cont}} (\mathbb{C}, + )^{{2\pi i \mathbb{Z}} }&  \rightarrow &\operatorname{Aut}_{\operatorname{cont}}\left( \frac{\mathbb{C}}{2\pi i \mathbb{Z}} , + \right) \\ & \phi & \mapsto & \widetilde{\phi} \end{array} .$$ 
We will now prove that $\Psi$ is a bijection. To do this, we will employ the concepts of algebraic topology. Let $\chi: \frac{\mathbb{C}}{2\pi i \mathbb{Z}} \rightarrow \frac{\mathbb{C}}{2\pi i \mathbb{Z}}$ be a continuous automorphism. Since the map $\operatorname{p}: \mathbb{C} \rightarrow \frac{\mathbb{C}}{2\pi i \mathbb{Z}}$ induces a universal cover of $\frac{\mathbb{C}}{2\pi i \mathbb{Z}}$ and $\pi_1(\mathbb{C}, 0)$ is trivial, we can apply \cite[Propositions 1.33 and 1.34]{Hatcher} to conclude that there exists a unique continuous map $\operatorname{Lift}(\chi): \mathbb{C} \rightarrow \mathbb{C}$ that sends $0$ to $0$ and satisfies $\operatorname{p} \circ \operatorname{Lift}(\chi) = \chi \circ \operatorname{p}$. To demonstrate that $\chi$ is a homomorphism, we can introduce the functor $\operatorname{Lift}$ from the category of pointed connected manifolds to the category of pointed connected manifolds. This functor on the objects assigns a pointed connected manifold to a pointed universal cover and on the morphism assigns a based continuous map between pointed manifolds to the unique continuous based lift between pointed universal covers. It can be proven that this functor $\operatorname{Lift}$ preserves the Cartesian product, which corresponds to the categorical product in the category of pointed manifolds. By observing that a functor preserving categorical products maps group objects to group objects and group morphisms to group morphisms, we can infer that $\operatorname{Lift}$ maps a group to a group and a group homomorphism to a group homomorphism. Therefore, since $\chi$ is a continuous automorphism, it follows that $\operatorname{Lift}(\chi)$ is also a continuous automorphism. The existence of $\operatorname{Lift}(\chi)$ implies the surjectivity of $\Psi$, and its uniqueness implies the injectivity of $\Psi$.

\noindent Finally, let $\phi \in \operatorname{Aut}_{\operatorname{cont}} (\mathbb{C}, +)^{{2\pi i \mathbb{Z}}}$. We have 
$$\begin{array}{lll} \Theta \circ \Psi (\phi)|_{\mathbb{C}^*} &=& \Theta(\widetilde{\phi})|_{\mathbb{C}^*}= \widetilde{\operatorname{exp}} \circ \widetilde{\phi} \circ \widetilde{\operatorname{exp}}^{-1} \\
&=& \widetilde{\operatorname{exp}} \circ \widetilde{\phi} \circ \operatorname{p} \circ \operatorname{L} = \widetilde{\operatorname{exp}} \circ \operatorname{p} \circ \phi \circ \operatorname{L} = \operatorname{exp} \circ \phi \circ \operatorname{L}.\end{array}$$
\end{proof} 
\subsubsection{Computing $\operatorname{Aut}_{\operatorname{cont}} (\mathbb{C}, + )^{{2\pi i \mathbb{Z}} } $}
Let us now compute the group of automorphisms of $\mathbb{C}$ that fix $2i\pi\mathbb{Z}$ set-wise.
\begin{lemm} \label{fixautoC}
We have 
$$\operatorname{Aut}_{\operatorname{cont}} (\mathbb{C}, + )^{{2\pi i \mathbb{Z}} } =\{ \phi_{\eta, \alpha} :   \eta \in \{ \pm 1\} , \!  \alpha \in \mathbb{C} \setminus i\mathbb{R} \} $$
where $\phi_{\eta, \alpha}$ is a map from $\mathbb{C}$ to $\mathbb{C}$ such that $\phi_{\eta, \alpha}(ai + b) = \eta ai +\alpha b$, for all $a, b \in \mathbb{R}$, where  $\eta \in \{ \pm 1\}$ and $\alpha \in \mathbb{C} \setminus i\mathbb{R}$. 

\end{lemm}
\begin{proof}

Let us consider an automorphism $\phi$ of $\mathbb{C}$ that fixes $2\pi i\mathbb{Z}$ set-wise. By $\S 3.1$, since $\phi$ is a continuous group automorphism, we can conclude that $\phi$ is $\mathbb{R}$-linear. Consequently, $\phi$ preserves $2\pi i\mathbb{Z}$ if and only if $\phi(i) \in\{ \pm i \}$.

Now, let us view $(\mathbb{C},+)$ as an $\mathbb{R}$-vector space with the basis $\{1, i\}$. This implies that the map $\phi$ is fully determined by the choice of $\phi(i)$ (either $+i$ or $-i$) and the value of $\phi(1)$, denoted by $\alpha$. For $\phi$ to be an automorphism, it is necessary that the set $\{\alpha, \phi(i)\}$ forms a basis for $\mathbb{C}$ over $\mathbb{R}$. Hence, we require $\alpha$ to belong to $\mathbb{C}\backslash i\mathbb{R}$. The converse is clear. 

\end{proof} 

\subsubsection{Computing the multiplicative automorphisms of $\mathbb{C}$}
Finally, let us deduce the form of the multiplicative continuous automorphisms of $\mathbb{C}$. From Lemma \ref{fixautoC} and Lemma \ref{gpautoC}, we obtain:
\begin{theo}\label{autoC}
The multiplicative automorphisms of $\mathbb{C}$ continuous over $\mathbb{C}^*$ (resp. $\mathbb{C})$ can be expressed precisely in one of two forms:
$$\begin{array}{cccl}  \epsilon_\alpha : & \mathbb{C} & \rightarrow & \mathbb{C} \\ 
& z=rs & \mapsto & r^\alpha s,
\end{array} 
\text{ \ }
\begin{array}{cccl}  \overline{\epsilon_\alpha} : & \mathbb{C} & \rightarrow & \mathbb{C} \\ 
& z=rs & \mapsto &r^\alpha\overline{s}
\end{array} $$
where $r \in \mathbb{R}_{>0}$, $s \in \mathbb{S}$ (the unit circle), and $\alpha \in \mathbb{C} \setminus i\mathbb{R}$ (resp. $\alpha \in \mathbb{C} \setminus i\mathbb{R}$ such that ${\sf Re} (\alpha) >0$).

Moreover, the following properties hold: for all $\alpha, \beta \in \mathbb{C} \setminus i\mathbb{R}$,
\begin{itemize}
\item $\overline{\epsilon_1}$ corresponds to complex multiplication.
\item $\overline{\epsilon_\alpha} = \epsilon_\alpha \circ \overline{\epsilon_1}$ and $\epsilon_\alpha = \overline{\epsilon_\alpha} \circ \overline{\epsilon_1}$.
\item $\overline{\epsilon_1} \circ \epsilon_\alpha = \overline{\epsilon_{\overline{\alpha}}}$ and $\overline{\epsilon_1} \circ \overline{\epsilon_\alpha} = \epsilon_{\overline{\alpha}}$.
\item $\epsilon_\alpha \circ \epsilon_\beta= \epsilon_{\operatorname{Re}(\beta) \alpha + i\operatorname{Im}(\beta)}$.
\item $\overline{\epsilon_\alpha} \circ \overline{\epsilon_\beta}=  \epsilon_\alpha \circ  \epsilon_{\overline{\beta}}= \epsilon_{ \operatorname{Re}(\beta) \alpha - i\operatorname{Im}(\beta)}$
\item $\overline{\epsilon_\alpha} \circ {\epsilon_\beta}= \overline{\epsilon_{ \operatorname{Re}(\beta) \alpha - i\operatorname{Im}(\beta)}} $
\item ${\epsilon_\alpha} \circ \overline{\epsilon_\beta}=\overline{\epsilon_{ \operatorname{Re}(\beta) \alpha + i\operatorname{Im}(\beta)}}$
\item The inverse of $\epsilon_\alpha$ (resp. $\overline{\epsilon_\alpha}$) is $\epsilon_{\frac{1 - i\operatorname{Im}(\alpha)}{\operatorname{Re}(\alpha)}}$ (resp. $\overline{\epsilon_{\frac{1 + i\operatorname{Im}(\alpha)}{\operatorname{Re}(\alpha)}}}$).
\item $\epsilon_\alpha$ is a homeomorphism over $\mathbb{C}^*$ (resp. $\mathbb{C}$).
\end{itemize}
\end{theo} 

\begin{rem}
When $K = \mathbb{C}$ and $\alpha \in \mathbb{R}^*$, the restriction ${\epsilon_\alpha|}_{\mathbb{R}}$ is identical to the automorphism $\epsilon_\alpha$ defined for $K = \mathbb{R}$. This consistency in notation is well-justified.
\end{rem}

\subsubsection{Identifying the group structure of the multiplicative automorphisms of $\mathbb{C}$}
By examining the properties of continuous multiplicative automorphisms of $\mathbb{C}$, we can deduce their group structure. Rewriting Theorem \ref{autoC}, we obtain:
\begin{prop} \label{groupaut}
We have the following canonical bijections
$$ \operatorname{Aut}_{\operatorname{cont}} (\mathbb{C}^*, \cdot) \simeq^{\chi_1} (\mathbb{C}\setminus i\mathbb{R} \times \{ \pm 1\},\star) \simeq^{\chi_2}   (\mathbb{R}^*
\times \mathbb{R}\times \{ \pm 1\}, \smallstar)$$
and 
$$ \operatorname{Aut}_{\operatorname{cont}} (\mathbb{C}^*, \cdot) \simeq^{\chi_2 \circ \chi_1}    (\mathbb{R}_{>0}\times \mathbb{R}\times \{ \pm 1\}, \smallstar)$$
where 
\begin{itemize} 
\item the group operation $\star$ on $\mathbb{C}\setminus i\mathbb{R} \times \{ \pm 1\}$ as follows: for any $(\alpha, u)$ and $(\beta,v)$ in $\mathbb{C}\setminus i\mathbb{R}\times \{ \pm 1\}$, the operation is defined as
$$(\alpha, u) \star (\beta, v) = (\operatorname{Re}(\beta) \alpha + u i\operatorname{Im}(\beta), uv).$$
The identity element of the group is $(1,1)$, and for any $(\alpha, u)$ in $\mathbb{C}\backslash i\mathbb{R}\times \{ \pm 1\}$, the inverse element is given by $\left( \frac{1 - iu \operatorname{Im}(\alpha)}{\operatorname{Re}(\alpha)}, u\right)$.
\item the group operation $\smallstar$ on $\mathbb{R}^*\times \mathbb{R}\times \{ \pm 1\}$ is defined as follows: For all $a, c\in\mathbb{R}^*$, $b, d\in \mathbb{R}$, and $u, v \in \{ \pm 1\}$, we have 
$$(a, b, u)\smallstar (c,d, v)=  (ac, cb + u d , uv).$$
The identity element for $\smallstar$ is $(1, 0,1)$ and the inverse of $(a, b, u)$ is $(\frac{1}{a} ,  -\frac{ub}{a}, u)$.
\item $\chi_1: \operatorname{Aut}_{\operatorname{cont}} (\mathbb{C}^*, \cdot) \rightarrow (\mathbb{C}\setminus i\mathbb{R} \times \{ \pm 1\},\star)$ sends $\epsilon_\alpha$ to $(\alpha, 1)$ and $\overline{\epsilon_\alpha}$ to $(\alpha, -1)$. 
\item $\chi_2: (\mathbb{C}\setminus i\mathbb{R}\times \{ \pm 1\}, \star)\rightarrow (\mathbb{R}^*\times \mathbb{R}\times \{ \pm 1\}, \smallstar)$ sends $(a+ib, u)$ to $(a, b, u)$.
\end{itemize}
\end{prop}
\begin{rem}
We note the following consequence of Proposition \ref{groupaut},
\begin{enumerate} 
\item $(\mathbb{C}\backslash i\mathbb{R}\times \{1\}, \star)$ is a normal subgroup of $(\mathbb{C}\backslash i\mathbb{R}\times \{\pm 1\}, \star)$, $(\{ 1\} \times \{ \pm 1\}, \star)$ is a subgroup of $(\mathbb{C}\backslash i\mathbb{R}\times \{\pm 1\}, \star)$, and the group $(\mathbb{C}\backslash i\mathbb{R}\times \{\pm 1\}, \star)$ is the internal semi-direct product of $(\mathbb{C}\backslash i\mathbb{R}\times \{1\}, \star)$ and $(\{ 1\} \times \{ \pm 1\}, \star)$.
\item $\{ (1, 0 , 1), (-1, 0,-1)\}$ is the centre of  $(\mathbb{R}^*\times \mathbb{R} \times \{\pm  1\}, \smallstar)$. 
\item $\mathbb{R}^*\times \{ 0\}\times \{  1\}$ is a subgroup of $\mathbb{R}^*\times \mathbb{R}\times \{ \pm 1\}$, $ \{  1\} \times \mathbb{R}\times \{ \pm 1\}$ is a normal subgroup of $\mathbb{R}^*\times \mathbb{R}\times \{ \pm 1\}$, and $\mathbb{R}^*\times \mathbb{R}\times \{ \pm 1\}$ is the internal semi-direct product of $ \{  1\} \times \mathbb{R}\times \{ \pm 1\}$ and $\mathbb{R}^*\times \{ 0\}\times \{  1\}$. 
\item We can endow $\mathbb{R}^*\times \mathbb{R}\times \{ \pm 1\}$ with the product topology, where the topology is defined component-wise using the natural topology of $\mathbb{R}$. With this topology, $(\mathbb{R}^*\times \mathbb{R}\times \{ \pm 1\}, \smallstar)$ becomes a topological group, and it has four connected components: $\mathbb{R}_{>0}\times \mathbb{R}\times \{ 1\}$, $\mathbb{R}_{<0}\times \mathbb{R}\times \{ 1\}$, $\mathbb{R}_{>0}\times \mathbb{R}\times \{ -1\}$, and $\mathbb{R}_{<0}\times \mathbb{R}\times \{ -1\}$. 
\end{enumerate}
\end{rem}
From Corollary \ref{cor1}, we deduce the set of field structures transported from $\mathbb{C}$ by multiplicative automorphisms of \(\mathbb{C}\) continuous over $\mathbb{C}^*$ (resp. $\mathbb{C}$) can be endowed with a group operation.
\begin{corollary} \label{isomautC}
We have
\[
(\mathcal{R}^{{\sf cont}}_{{\sf m}}(\mathbb{C}^*), \oast) \simeq  (\operatorname{Aut}_{\operatorname{cont}}(\mathbb{C}^*, \cdot)/\langle \overline{\epsilon_1} \rangle, \bigostar) \simeq^\Phi (\mathbb{R}^*\times \mathbb{R} \times \{ 1\}, \smallstar)
\]
and
\[
(\mathcal{R}^{{\sf cont}}_{{\sf m}}(\mathbb{C}), \oast) \simeq (\operatorname{Aut}_{\operatorname{cont}}(\mathbb{C}, \cdot)/\langle \overline{\epsilon_1} \rangle, \bigostar) \simeq^\Phi  (\mathbb{R}_{>0}\times \mathbb{R} \times \{  1\}, \smallstar)
\]
where 
\begin{itemize}
\item $\mathcal{R}^{{\sf cont}}_{{\sf m}}(\mathbb{C}^*)$ (resp. $\mathcal{R}^{{\sf cont}}_{{\sf m}}(\mathbb{C})$) is the set of field structures transported from $\mathbb{C}$ by multiplicative automorphisms of \(\mathbb{C}\) continuous over $\mathbb{C}^*$ (resp. $\mathbb{C}$). and the operation $\oast$ is defined by $[(\mathbb{C} , +_{\epsilon_\alpha}, \cdot) ] \oast [(\mathbb{C} , +_{\epsilon_\beta}, \cdot)] = [(\mathbb{C} , +_{\epsilon_{\alpha\star \beta}}, \cdot)]$, for all $\alpha, \beta \in \mathbb{C} \setminus i \mathbb{R}$. 
\item $\operatorname{Aut}_{\operatorname{cont}}(\mathbb{C}, \cdot)/\langle \overline{\epsilon_1} \rangle  $ is the quotient group of $\operatorname{Aut}_{\operatorname{cont}}(\mathbb{C}, \cdot)$ by the subgroup $\langle \overline{\epsilon_1} \rangle$ acting on the right, the operation $\bigostar$ between the classes $[\epsilon_\alpha]$ and $[\epsilon_\beta]$ is defined as $[\epsilon_\alpha] \bigostar [\epsilon_\beta] = [\epsilon_{\alpha \filledstar \beta}]$, for all $\alpha, \beta \in \mathbb{C}\setminus i\mathbb{R}$. 
 We recall that $\operatorname{Aut}(\mathbb{C}, +, \cdot)=\langle \overline{\epsilon_1} \rangle$. 
 \item $\Phi: (\mathbb{R}^*\times \mathbb{R} \times \{ 1\}, \smallstar) \rightarrow (\operatorname{Aut}_{\operatorname{cont}}(\mathbb{C}, \cdot)/\langle \overline{\epsilon_1} \rangle, \bigostar)$ sending $(a, b, 1)$ to $[\epsilon_{a+ib}]$
\end{itemize}
\end{corollary}

\begin{rem}
We note that the operation $\bigostar$ in Corollary \ref{isomautC} is defined by explicitly choosing a representative from the coset of $\operatorname{Aut}_{\operatorname{cont}}(\mathbb{C}, \cdot)/\langle \overline{\epsilon_1} \rangle  $ to construct the operation. This operation does not correspond to the class of the composition of the maps $\epsilon$'s. In other words, we are simply transferring the group law $\filledstar$ on $\mathbb{C}\setminus i\mathbb{R}$ to $\operatorname{Aut}_{\operatorname{cont}}(\mathbb{C}, \cdot)/\langle \overline{\epsilon_1} \rangle$ through the bijective map $\Lambda$.
\end{rem}
We can describe the set of multiplicative automorphism of $\mathbb{C}$ continuous over $\mathbb{C}^*$ as another semi-direct product as follows. We set 
$$\operatorname{Aut} (\mathbb{C}, \cdot )^{\langle \overline{\epsilon_1}\rangle }  = \{ \phi \in\operatorname{Aut} (\mathbb{C}, \cdot ) | \overline{\epsilon_1} \circ \phi= \phi \circ \overline{\epsilon_1}\}.$$ We obtain:
\begin{itemize}
\item $(\operatorname{Aut} (\mathbb{C}, \cdot )^{\langle \overline{\epsilon_1} \rangle}, \circ)$ is a subgroup of $(\operatorname{Aut} (\mathbb{C}, \cdot ), \circ)$.

\item 
$$ \begin{array}{lll} \operatorname{Aut} (\mathbb{C}, \cdot)^{\langle \overline{\epsilon_1} \rangle} 
 &=&  \{ \phi \in \operatorname{Aut} (\mathbb{C}, \cdot) | \phi (\mathbb{R} ) = \mathbb{R} \}\\ 
&=& \{ \phi \in \operatorname{Aut} (\mathbb{C}, \cdot) | \phi (\mathbb{R}_{>0} ) = \mathbb{R}_{>0}  \}\\ 
&\subseteq &  \{ \phi \in \operatorname{Aut}(\mathbb{C}, \cdot) | \phi (\mathbb{S} ) = \mathbb{S}  \}\\
&\simeq&  (\operatorname{Aut} (\mathbb{R}, \cdot ), \circ) \simeq (\operatorname{Aut} (\mathbb{R}_{>0}, \cdot ), \circ) \simeq  (\operatorname{Aut} (\mathbb{S}, \cdot ), \circ).
 \end{array} $$
 
 \item the group isomorphisms $$( \mathbb{R}^*, . )  \simeq (\operatorname{Aut}_{\operatorname{cont}} (\mathbb{C}^*,\cdot)^{\langle \overline{\epsilon_1} \rangle}, \circ) \text{ and } ( \mathbb{R}^*_{>0}, . )  \simeq (\operatorname{Aut}_{\operatorname{cont}} (\mathbb{C},\cdot )^{\langle \overline{\epsilon_1} \rangle}, \circ).$$ 
 sending $\alpha \in\mathbb{R}^*$ to $\epsilon_\alpha$. 
\end{itemize} 
Now, the set 
\[
R_{\operatorname{cont}}(\mathbb{C}, \cdot) = \{ \phi \in \text{Aut}_{\text{cont}}(\mathbb{C},\cdot) \mid \forall z \in \mathbb{C}, \exists s \in \mathbb{S}, \phi(z) = zs \}
\]
is a normal subgroup of $\text{Aut}_{\text{cont}}(\mathbb{C},\cdot)$ isomorphic to $\{1\} \times \mathbb{R} \times \{\pm 1\}$, and $\text{Aut}_{\text{cont}}(\mathbb{C},\cdot)$ is the internal semi-direct product of $R_{\text{cont}}(\mathbb{C}, \cdot)$ by $\text{Aut}(\mathbb{C},\cdot)^{\langle \overline{\epsilon_1} \rangle}$.

\subsection{Complex fields induced by bijective map}
We can extend the notions of the real part, imaginary part, modulus, and complex conjugate to the continuous complex field induced by a multiplicative automorphism (see also Appendix). More precisely, given $\phi$ a bijection from $\mathbb{C}$ to $\mathbb{C}$. 
\begin{itemize} 
\item We denote $ \mathbb{R}_\phi$ for the field $( \phi^{-1} (\mathbb{R}), +_\phi, \cdot_\phi)$. $ \mathbb{R}_\phi$ is a subfield of $\mathbb{C}_\phi$ isomorphic to $(\mathbb{R}, + , \cdot)$. We have $[ \mathbb{C}_\phi :\mathbb{R}_\phi] =2$. Moreover, $\{1, \ i_\phi\}$ is a basis for $\mathbb{C}_\phi$ over $\mathbb{R}_\phi$. When $\phi$ induces an endo-bijection on $\mathbb{R}$,  $\phi^{-1} (\mathbb{R})=\mathbb{R}$ showing that the notation is well-defined.
\item We define the $\phi$-real part of $z$, denoted by $\operatorname{Re}_\phi(z)$, to be 
$$\operatorname{Re}_\phi(z)= \phi^{-1} ( \operatorname{Re} ( \phi (z) ))$$ and the $\phi$-imaginary part of $z$, denoted by $\operatorname{Im}_\phi(z)$, to be $\operatorname{Im}_\phi(z) =\phi^{-1} (\operatorname{Im} ( \phi (z) ))$. We have 
$$\operatorname{Re}_\phi(z), \operatorname{Im}_\phi (z) \in \mathbb{R}_\phi$$ and for any $z \in \mathbb{C}$, $ z = \operatorname{Re}_\phi(z) +_\phi i_\phi  \operatorname{Im}_\phi(z) .$
\item We define the $\phi$-conjugate of $z$ to be $\overline{z}^{\phi} = Re_\phi ( z) -_\phi i_\phi  Im_\phi (z)  $.
\item We define the $\phi$-modulus of $z$ to be $|z|_\phi  = \phi^{-1} (|\phi (z) |) $ which is an element $\phi^{-1} (\mathbb{R}_{>0})$.  
\end{itemize} 
In particular, the following diagram:
\[
\xymatrix{
\mathbb{C} \ar[rr]^{\phi^{-1}}_{\cong} \ar[d]^{\iota}_{\cong} && \mathbb{C}_\phi \ar[d]^{\iota_\phi}_{\cong} \\
\mathbb{R}^2 \ar[rr]^{\phi^{-1} \times \phi^{-1}}_{\cong} && \mathbb{R}_\phi^2
}
\]
is commutative, where the mappings are defined as follows: for any $z\in \mathbb{C}$,
\begin{itemize} 
\item $\iota(z) = (\operatorname{Re}(z), \operatorname{Im}(z))$,
\item $\iota_\phi(z) = (\operatorname{Re}_\phi(z), \operatorname{Im}_\phi(z))$,
\end{itemize} 
Moreover, as indicated in the diagram, $\phi^{-1}$, $\iota$, $\iota_\phi$, and $\phi^{-1} \times \phi^{-1}$ are isomorphisms. 
For all $z \in \mathbb{C}$, we have the following fundamental identities:
\begin{enumerate}
	\item $\phi(\overline{z}^\phi) = \overline{\phi(z)}$,
	\item $|z|_\phi^2 = z \overline{z}^\phi = \operatorname{Re}_\phi(z)^2 +_\phi \operatorname{Im}_\phi(z)^2$,
	\item $\operatorname{Re}_\phi(z) = \frac{z +_\phi \overline{z}^\phi}{2_\phi}$,
	\item $\operatorname{Im}_\phi(z) = \frac{z -_\phi \overline{z}^\phi}{(2i)_\phi}$,
	\item $|\phi^{-1}(e^{i\theta})|_\phi = 1$,
	\item $z = |z|_\phi \phi^{-1}(e^{i\operatorname{arg}(\phi(z))})$.    
\end{enumerate}

We can even introduce the concepts of cosine and sine functions within the present context.
More precisely, for any $\alpha \in \mathbb{C}\backslash i \mathbb{R}$.
\begin{enumerate}
\item we denote $ \operatorname{cos}_\alpha(z):= \operatorname{cos}( \operatorname{arg}(\epsilon_\alpha(z)))$ and $ \operatorname{sin}_\alpha(z):= \operatorname{cos}( \operatorname{arg}(\epsilon_\alpha(z)))$. 
\item we denote $\overline{ \operatorname{cos}_\alpha}(z):= \operatorname{cos}( \operatorname{arg}(\overline{ \epsilon_\alpha}(z)))$ and $ \overline{ \operatorname{sin}_\alpha}(z):= \operatorname{cos}( \operatorname{arg}(\overline{ \epsilon_\alpha}(z)))$. 
\end{enumerate}
In particular, we have, for any $\alpha \in \mathbb{C}\backslash i \mathbb{R}$ and $z\in \mathbb{C}$, 
\begin{enumerate} 
\item We have $\epsilon_\alpha^{-1} ( \mathbb{R}) = \mathbb{R}^{\frac{ 1- i \operatorname{Im} (\alpha) }{ \operatorname{Re} (\alpha)}} $ and $\overline{\epsilon_\alpha}^{-1} ( \mathbb{R}) = \mathbb{R}^{\frac{ 1+ i \operatorname{Im} (\alpha) }{ \operatorname{Re} (\alpha)}} $.
\item $i_{\epsilon_\alpha} =i$, $\operatorname{arg}(\epsilon_\alpha(z))= \operatorname{Im} (\alpha) \operatorname{ln} |z|+  \operatorname{arg}(z)$, $ |z|_{\epsilon_\alpha} =|z|^{1-iIm (\alpha)} $
$$\operatorname{Re}_{\epsilon_\alpha} ( z)= |z|^{1-i \operatorname{Im}(\alpha) }  \operatorname{sgn} (\operatorname{cos}_\alpha(z)) |\operatorname{cos}_\alpha(z)|^{\frac{ 1- i \operatorname{Im}(\alpha) }{\operatorname{Re}(\alpha) }} ,$$
$$ \operatorname{Im}_{\epsilon_\alpha} ( z)=  |z|^{1-i \operatorname{Im}(\alpha) }  \operatorname{sgn} (\operatorname{sin}_\alpha(z)) |\operatorname{sin}_\alpha(z)|^{\frac{ 1- i \operatorname{Im}(\alpha)  }{\operatorname{Re}(\alpha)}}. $$ 
\item $i_{\overline{\epsilon_\alpha}} =-i$, $\operatorname{arg}(\overline{\epsilon_\alpha} (z))= \operatorname{Im} (\alpha) \operatorname{ln} |z|-  \operatorname{arg}(z)$, $ |z|_{\overline{\epsilon_\alpha}} =|z|^{1+i\operatorname{Im}(\alpha)} $
$$\operatorname{Re}_{\overline{\epsilon_\alpha}} ( z)= |z|^{1+i \operatorname{Im}(\alpha) }  \operatorname{sgn} (\overline{ \operatorname{cos}_\alpha}(z)) |\overline{ \operatorname{cos}_\alpha}(z)|^{\frac{ 1+ i \operatorname{Im}(\alpha) }{\operatorname{Re}(\alpha) }} ,$$
$$ \operatorname{Im}_{\overline{\epsilon_\alpha}} ( z)=  |z|^{1+i \operatorname{Im}(\alpha) }  \operatorname{sgn} (\overline{ \operatorname{sin}_\alpha}(z)) |\overline{ \operatorname{sin}_\alpha}(z)|^{\frac{ 1+ i \operatorname{Im}(\alpha)  }{\operatorname{Re}(\alpha)}}. $$ 
\end{enumerate}

\section{Characterizing additive structures on a fixed scalar group}\label{add_struct}
In this section, we go a bit further with the question of understanding additive structures over a given multiplicative group. This question is motivated by the theory of near-vector spaces. 
We recall that a scalar group $(F, \cdot, 1, 0, -1)$ is a monoid $(F\setminus \{ 0 \} , \cdot, 1)$ such that $(F\setminus \{ 0 \} , \cdot, 1)$ is a group, $0, -1, 1 \in F$, for all $\alpha \in F$, $\alpha \cdot 0 = 0 \cdot \alpha =0$ and $\pm 1 $ is the solution set of the equation $x^2 -1$. Indeed, in order to understand all the near-vector spaces over the scalar group $(F, \cdot)$, we need to understand all the additive binary operations, $+$, such that $(F, +,  \cdot)$ is a near-field.  

Given $(F, \cdot, 1, 0, -1)$ a scalar group and $(F, +, \cdot)$ a left-near field. 
Then 
\begin{itemize} 
\item $0$ is the zero element with respect to $+$. (Indeed, it is the only non-invertible element.)
\item Given $\alpha\in F$, the additive inverse of $\alpha$ with respect to $+$ is either $(-1)\cdot \alpha$. Moreover, for all $\alpha, \beta \in F$, we have $\alpha \cdot (-1) = (-1) \cdot \alpha$. (To see this, we denote by $\boxminus 1$ the additive inverse of $1$, the result follows from the equality $(\boxminus 1)^2=1$ which is deduced from the equality  $0 =(\boxminus  1) (1+ (\boxminus  1 ))= (\boxminus  1) + ( \boxminus  1)^2 $.  Indeed, it follows that either $\boxminus  1=1$ or $\boxminus  1=-1$. Noting that when $\boxminus  1=1$, then $1 + 1 =0$, that is ${\sf Char}(F)=2$, we deduce the result. Finally, to see that $\alpha\cdot (-1) = (-1)\cdot \alpha$, for all $\alpha\in F$, it suffices to observe that $(\alpha\cdot (-1) \cdot\alpha^{-1})^2=1$, for all $\alpha \in F^*$.)
\end{itemize}

Given a scalar group, we observe that using the distributivity of the multiplication on the addition it is enough to know $1+ \alpha$ for all $\alpha\in F$, to fully determine completely an addition $+$ such that $(F, + , \cdot)$ is a near-field. This data can be viewed as the map $\rho: F \rightarrow F$ sending $\alpha$ to $1+\alpha$. The next definition identifies the properties that such a map needs to satisfy to indeed define an addition. 
\begin{defin} \label{nfam}
Let $(F, \cdot)$ be a scalar group.
 A map $\rho: F \rightarrow F$ is said to be {\sf a near-field addition map on $(F, \cdot)$} if for all $\alpha \in F^*$ and $\beta \in F$, \\
$1.$ $\rho (0) =1$. We refer to this property as the identity property of $\rho$.\\
$2.$ $\rho (-1)=0$. We refer to this property as the inverse property of $\rho$.\\
$3.$ $\rho ( \alpha^{-1} )= \alpha^{-1} \rho (\alpha)$. We refer to this property as the abelian property of $\rho$. \\
$4.$ $\rho ( \alpha \rho (\beta)) = \alpha \rho ( \beta \rho ( (\alpha \beta)^{-1}))$,
 when $\alpha , \ \beta \in F^*$. We refer to this property as the associative property of $\rho$.
\end{defin} 

\begin{rem} \label{assoc}
Let $(F, \cdot)$ be a scalar group, $\rho$ be a near-field addition map, and $ \alpha \in F^* \backslash \{ -1\}$. Then we have $\rho ( \alpha \rho (\alpha^{-1} \beta)) = \rho ( \alpha) \rho (   \rho ( \alpha)^{-1} \beta)$.
\end{rem}

\begin{deflem}\label{nfa}
Let $(F, \cdot)$ be a scalar group.
\begin{enumerate} 
\item Given a near-field $(F, +, \cdot)$, we denote $\rho_{+}$ the map sending $\alpha$ to $ 1+ \alpha$. $\rho_{+}$ is a near-field addition map. 
\item Given a near-field addition map $\rho: F \rightarrow F$, we define the operation $+_\rho$ as the operation defined for all $\alpha, \beta \in F$ such that $\alpha +_\rho \beta = \alpha \rho (\alpha^{-1} \beta) $ when $\alpha\neq 0$  and $\alpha +_\rho \beta = \beta$, otherwise. $(F, +_\rho, \cdot)$ is a near-field. We denote $(F, +_\rho, \cdot)$ simply as ${}_\rho F$ and refer to it as a $\rho$-near-field. When $\alpha_1, \cdots, \alpha_s\in F$, we denote ${}^\rho \sum_{k=1}^s \alpha_k = \alpha_1 +_\rho \cdots  +_\rho \alpha_s$. 
\end{enumerate}
\end{deflem} 
\begin{proof}
Let $(F, \cdot)$ be a scalar group. 
\begin{enumerate}
\item Suppose $(F , + , \cdot)$ is a left near field. Let $\mathcal{A} : F \times F \rightarrow F$  be the binary operation sending $(\alpha , \beta)$ to $\alpha + \beta$ and $\rho = \mathcal{A} ( 1,-)$.  From the observation made at the beginning of the section, we have $\rho(0) =\mathcal{A} ( 1,0)= 1+ 0 = 1$, and  $ \mathcal{A} (1, -1) = \rho(-1) = 0$. Furthermore, the commutativity of $\mathcal{A}$ leads to $\rho(\alpha^{-1}) = \mathcal{A} (1, \alpha^{-1}) = \alpha^{-1}  \mathcal{A} (\alpha, 1) = \alpha^{-1}  \mathcal{A} (1, \alpha) = \alpha^{-1} \rho(\alpha)$, for any $\alpha$ non-zero element of $F$. 
 Given $a, b , c\in F^*$. The associativity property means that $\mathcal{A}(\mathcal{A}(a, b), c) = \mathcal{A}(a, \mathcal{A}(b, c))$. The left-hand side gives
\bea
\mathcal{A}(\mathcal{A}(a, b), c) &=& c \mathcal{A}(1, c^{-1}  a\rho(a^{-1} b)) = c \rho(c^{-1} a\rho(a^{-1} b)).
\eea
The right-hand side gives
\bea
\mathcal{A}(a, \mathcal{A}(b, c))=  a \mathcal{A}(1, a^{-1} \rho(b^{-1} c)) 
= a \rho(a^{-1} b\rho(b^{-1} c)).
\eea
Thus,\begin{equation}  \rho(c^{-1} a  \rho(a^{-1}  b)=c^{-1}  a \rho(a^{-1} b \rho(b^{-1} c)).\end{equation} Setting $\alpha = c^{-1} a$ and $\beta = a^{-1} b$, we have $\alpha \beta = c^{-1} b$ and thus $(\alpha \beta)^{-1} = b^{-1}  c$. 
 Substituting these into Equation (1) we obtain the result $\rho ( \alpha \rho (\beta)) = \alpha \rho ( \beta \rho ( (\alpha \beta)^{-1})$. \\ 
\item Suppose that $\rho$ is a near-field addition map. 
It is clear that $F$ is closed under the operation $+_\rho$. 
By definition of a near-field addition map we know that $\rho$ has the properties $\rho (0) =1$, $\rho (-1)=0$, $\rho ( \alpha^{-1} )= \alpha^{-1} \rho (\alpha)$ and  $\rho ( \alpha \rho (\beta)) = \alpha \rho ( \beta \rho ( (\alpha \beta)^{-1})$, for all $\alpha, \beta \in F^*$.  
For any $\alpha,\ \beta, \ \gamma \in F^*$,\\  
$$\alpha +_\rho  \beta =  \alpha \rho(\alpha^{-1} \beta )=\alpha \alpha^{-1} \beta \rho(\beta^{-1} \alpha )= \beta \rho(\alpha \beta^{-1} )=  \beta +_\rho  \alpha$$
by the abelian property of $\rho$. 
Moreover, for any $\alpha \in F^*$, $$\alpha +_\rho 0  = \alpha \rho(\alpha^{-1} 0) =\alpha \rho(0)=  \alpha.$$ and $0 +_\rho \alpha =\alpha$, for all $\alpha \in F^*$. Thus, 0 is the zero element. Moreover, for any $\alpha \in F^*$, 
$$\alpha +_\rho (-\alpha) = \alpha \rho(-1)  = \alpha  0 = 0.$$ The result is still valid for $\alpha =0$. This shows that each $\alpha \in F$ has an additive inverse $-\alpha$. \\
Let $a, b \in F^*$ and $c \in F$. We now prove the associativity of $+_\rho$ 
$$\begin{array}{lll} (a+_\rho b) +_\rho c &=& a \rho(a^{-1} b) +_\rho c  =   c   +_\rho a \rho(a^{-1} b) = c \rho( c^{-1}  a  \rho(a^{-1}  b))\\
&=& a \rho(a^{-1} b \rho(b^{-1} c))= a+_\rho (b +_\rho c)
\end{array} $$
from the associative property of $\rho$ applied to $\alpha = c^{-1} a$ and $\beta = a^{-1} b$. When $a=0$ or $b=0$ the result is clear. 
For the distributivity of $\cdot$ over $+_\rho$, let $a , \alpha , \beta \in F^*$. We have $ a (\alpha \rho (\alpha^{-1} \beta))= a (\alpha +_\rho \beta )$ and $ (a \alpha)  \rho (\alpha^{-1} \beta) = a \alpha +_\rho a \beta $. We get $a (\alpha +_\rho \beta )= a \alpha + _\rho a \beta$. When either $a=0$ or $\alpha=0$ and $\beta=0$, the result is clear. This concludes the proof that $F$ is a near-field.
\end{enumerate}
\end{proof}
\begin{rem} \label{rho0} Let $(F, \cdot)$ be a scalar group and $\rho$ be a near-field addition map. 
\begin{enumerate} 
\item By the means of an induction, we prove that $\rho^n(0) \rho ( \rho^n(0)^{-1} \rho^m(0))=\rho^{n+m} (0) $, for all $n, \ m \in \mathbb{N}$. 
\item For all $n \in \mathbb{N}$, $\rho^n(0) \rho ( - \rho^n(0)^{-1} \rho^n(0) )=\rho^n(0) \rho ( - 1 )=0$
\item A near-field addition map is bijective. More precisely, if $\rho$ is a near-field addition map then its inverse is the map sending $\alpha$ to $-\rho(-\alpha)$. In particular, $\rho (\alpha ) \neq 0$ for any $\alpha \neq -1$.
\item For every $\alpha \in F \backslash \{ -1\}$, we have $\rho (\alpha)\in F^*$.  
\item Let $\alpha, \beta \in F$ and $n \in\mathbb{N}$.  By the means of induction, we can prove the following
 $$ {}^\rho \! \sum_{k=1}^n \alpha =  \alpha \rho^{n}(0) \text{ and }  {}^\rho \! \sum_{k=1}^n \alpha  \,+_\rho \,{}^\rho \! \sum_{k=1}^n \beta =  (\alpha +_\rho \beta) \rho^{n}(0) .$$ 
\end{enumerate}
\end{rem}
We next define the notion of a characteristic map. This map will permit us to define the characteristic of a near-field. Using Remark \ref{rho0} and Remark \ref{assoc}, we obtain as, expected and is known, that the characteristic of a near-field is either $0$ or a prime number. The characteristic map defines an embedding of a prime field onto any near-field. 
\begin{deflem} \label{charac} \label{rhocha}
 Let $(F, \cdot)$ be a scalar group and $\rho$ be a near-field addition map. 
 We define the $\rho$-characteristic map, denoted $\chi_\rho$, to be the map 
$$ \begin{array}{llll} \chi_\rho: & \mathbb{Z} &\rightarrow &F\\
& n &\mapsto &  \operatorname{sgn} (n)\rho^{|n|}(0). \end{array}
$$
where $\rho^{|n|}$ denotes $\rho$ composite with itself $|n|$ times when $n\neq 0$ and we set $\rho^{0}= \operatorname{id}$. We also denote $C_\rho$ the image of the map $\rho$. 
 Then $\chi_\rho$ is a ring homomorphism from $(\mathbb{Z}, +, \cdot )$ to $(F, +_\rho, \cdot)$. $(C_\rho , +_\rho)$ is a cyclic group isomorphic either to $\mathbb{Z}$, when $\chi_\rho$ is one-to-one, or $\mathbb{F}_p$ for some prime number $p$ otherwise. 
 When $
 \chi_\rho$ is one-to-one, $ \chi_\rho$ naturally induces a field morphism from $( \mathbb{Q}, +, \cdot)$ to $(F, +_\rho, \cdot)$
 $$\begin{array}{llll} \widetilde{\chi_\rho} : & \mathbb{Q} & \rightarrow  & F\\ & \frac{n}{m} & \mapsto & \operatorname{sgn} (nm)\rho^{|n|}(0) (\rho^{|m|}(0))^{-1}
 \end{array} $$ 
  localizing at the prime ideal $(0)$ of $\mathbb{Z}$. 
 Otherwise, $ \chi_\rho$ naturally induces a field morphism from $( \mathbb{F}_p, +, \cdot)$ to $(F, +_\rho, \cdot)$
  $$\begin{array}{llll} \widetilde{\chi_\rho} : & \mathbb{F}_p & \rightarrow  & F\\ & [n]_p & \mapsto & \operatorname{sgn} (n)\rho^{|n|}(0)
 \end{array} $$
using the first isomorphism theorem.  The {\sf characteristic $p$} of $\rho$ denoted $char(F)$ is $0$ when $\chi_\rho$ is one-to-one and $p$ when $ker(\chi_\rho)=p \mathbb{Z}$.  We denote $F_\rho$ the image of $\widetilde{\chi_\rho}$. Moreover, $F$  is a $F_\rho$-vector space and $F_p$ distributes on any element of $F$ and $F_\rho^*$ is a commutative multiplicative subgroup of $F^*$. 
\end{deflem}

When $F$ is finite near-field, since $F$ is a $F_p$-vector space, we have $F \simeq  \mathbb{F}_p^n$ where $p = {\sf Char}(F)$ and $n$ is the integer such that $|F|=p^n$. This proves the uniqueness of the additive map up to additive isomorphism.
From Proposition \ref{bij+} and \S 3.1, the additions on $\mathbb{F}_{p^n}$ are given by $ ( \alpha^{a} + \beta^{a})^{a^{-1_n}}$ where $a \in U_{p^n-1}$ and $a^{-1_n}$ is a representative for the multiplicative inverse of $a$ in $U_{p^n-1}$.

We have seen that depending on the field structure put on $\mathbb{Q}$, $\mathbb{Q}$ might not be a prime field (see end of section $\S 3.2$). The next result characterize the possible field structures on the scalar group $(\mathbb{Q}, \cdot)$.
 \begin{prop}
There exists an addition $\boxplus$, such that $(\mathbb{Q}, \boxplus, \cdot)$ is a field isomorphic to a valuation field $(K , + , \cdot)$ if and only if  $\mathcal{O}_K$ is a unique factorization domain with unit group $\{\pm 1\}$ and the cardinality of the set of its irreducible elements has the same cardinality as $\mathbb{N}$.
 \end{prop}
 \begin{proof}
Suppose that there exists an addition $\boxplus$, and an isomorphism $\Psi: (\mathbb{Q}, \boxplus, \cdot)\rightarrow (K , + , \cdot)$. Since $\Psi$ is a multiplicative automorphism, it sends $\pm 1$ to $\pm 1$, irreducible element to irreducible element, and $0$ to $0$. For the converse, any bijection from the set of prime number to the irreducible elements of $\mathcal{O}_K$ up to unit, extended by multiplication, sending $\pm 1$ to $\pm 1$ and $0$ to $0$ defines a field isomophism. 
 \end{proof}
From the previous lemma, we observe that \((\mathbb{F}_3(x), +, \cdot)\) is another field for which there exists an alternative addition \(\boxplus\) and an isomorphism \(\Psi: (\mathbb{Q}, \boxplus, \cdot) \rightarrow (\mathbb{F}_3(x), +, \cdot)\) (see also end of \S 2.1). It would be of considerable interest to determine which fields are isomorphic under a change in the addition operation, as these fields must exhibit very similar arithmetic properties.

\newpage
\section*{Appendix} 
The set of figures below illustrates the complex plane coordinate lines transported by multiplicative automorphisms of the form $\epsilon_\alpha$. 

\begin{center}
\begin{multicols}{2}
\begin{center}
\includegraphics[scale=0.43]{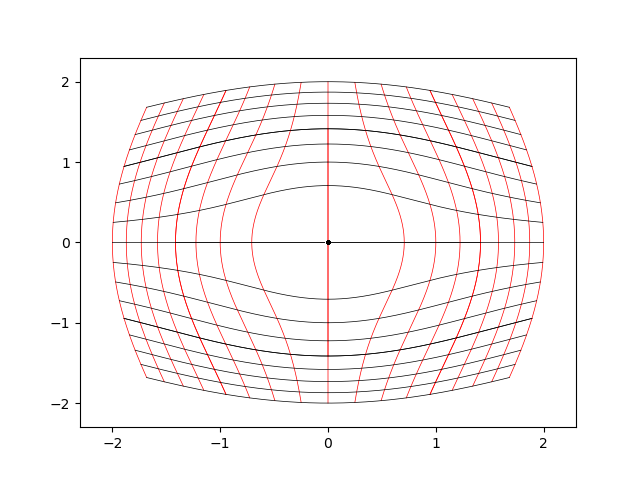}\\
$\alpha=\frac{1}{2}$
\end{center}
\columnbreak
\begin{center}
\includegraphics[scale=0.43]{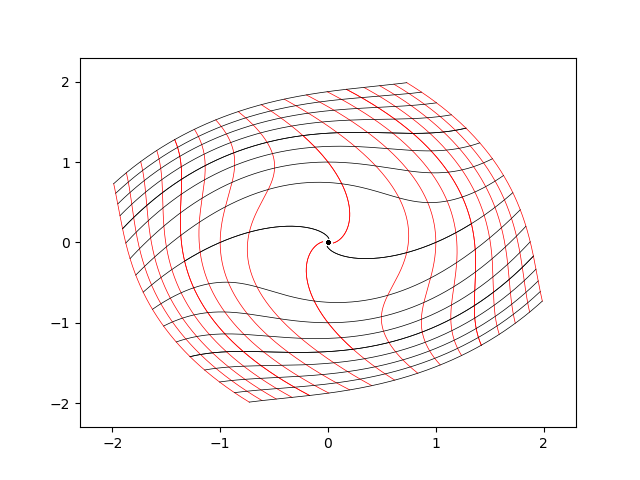}\\
$\alpha=\frac{1}{2}\left(\frac{\sqrt{3}}{2} + \frac{1}{2}i\right)$
\end{center}
\end{multicols}

\begin{multicols}{2}
\begin{center}
\includegraphics[scale=0.43]{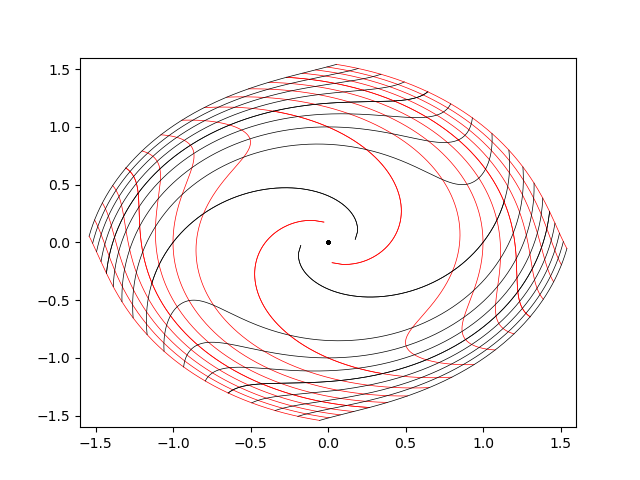}\\
$\alpha=\frac{1}{2}\left(\frac{1}{2} + \frac{\sqrt{3}}{2}i\right)$
\end{center}
\columnbreak
\begin{center}
\includegraphics[scale=0.43]{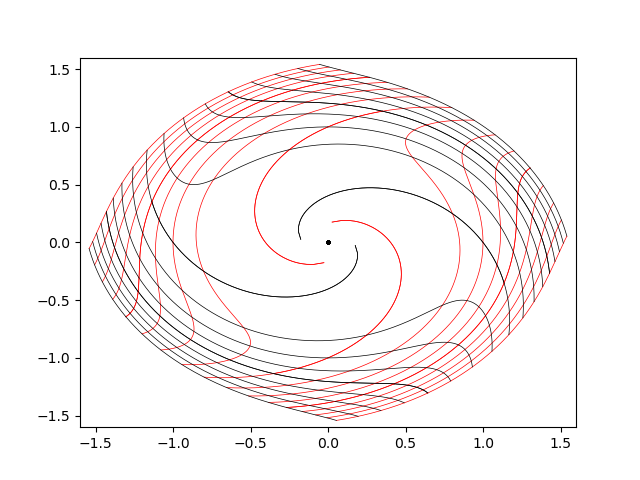}
$\alpha=\frac{1}{2}\left(\frac{1}{2} - \frac{\sqrt{3}}{2}i\right)$
\end{center}
\end{multicols}
\end{center}
\newpage

\begin{multicols}{2}
\begin{center}
\includegraphics[scale=0.4]{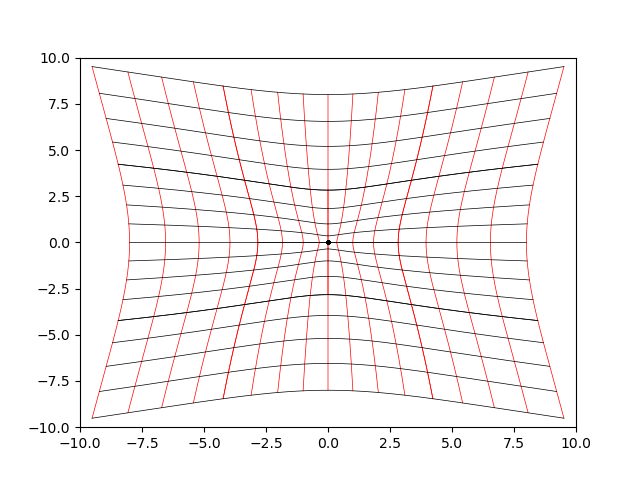}
$\alpha=\frac{3}{2}$
\end{center}
\columnbreak
\begin{center}
\includegraphics[scale=0.4]{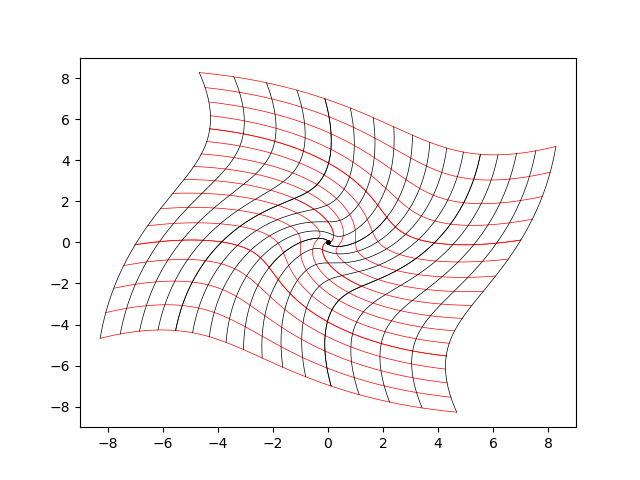}\\
$\alpha=\frac{3}{2}\left(\frac{\sqrt{3}}{2} + \frac{1}{2}i\right)$
\end{center}
\end{multicols}

\begin{multicols}{2}
\begin{center}
\includegraphics[scale=0.4]{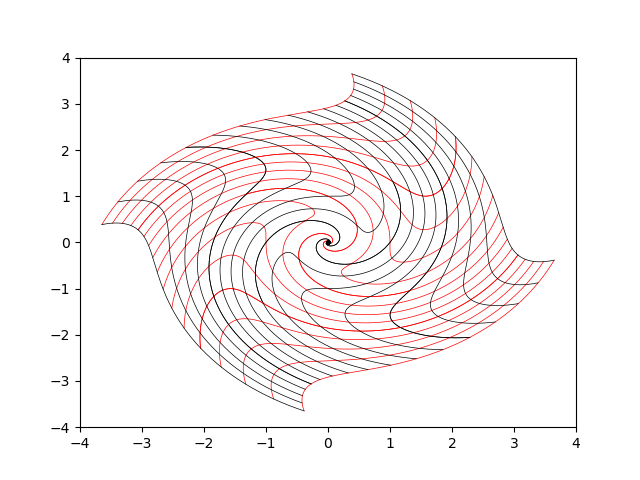}
$\alpha=\frac{3}{2}\left(\frac{1}{2} + \frac{\sqrt{3}}{2}i\right)$
\end{center}
\columnbreak
\begin{center}
\includegraphics[scale=0.4]{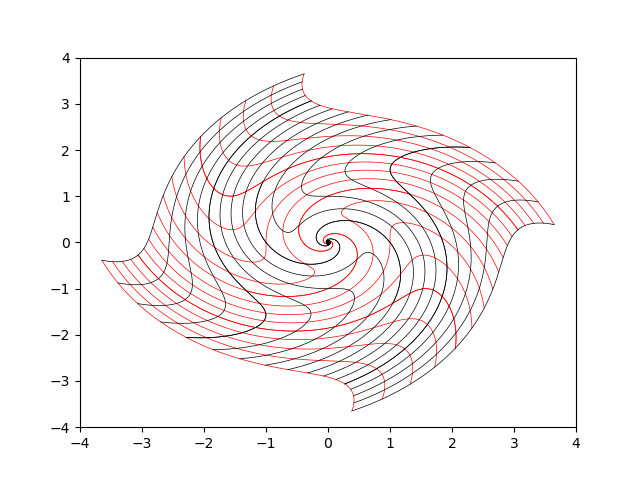}\\
$\alpha=\frac{3}{2}\left(\frac{1}{2} - \frac{\sqrt{3}}{2}i\right)$
\end{center}
\end{multicols}

\begin{center}
\includegraphics[scale=0.4]{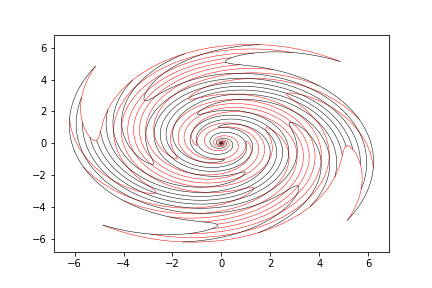}\\
$\alpha= 1+4i$
\end{center}

\begin{multicols}{2}
\begin{center}
\includegraphics[scale=0.4]{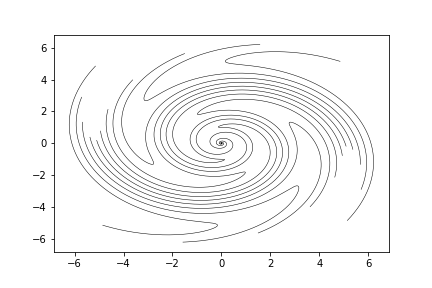}
Horizontal lines alone.
\end{center}
\columnbreak
\begin{center}
\includegraphics[scale=0.4]{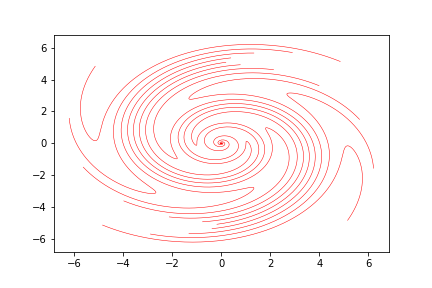}
Vertical lines alone.
\end{center}
\end{multicols}

If we write $\alpha = a + ib$, the figures above are all for the case where $a >0$. The case where $a < 0$ requires some explanation. The general form of the transformation is $z' = r^{a} e^{i(b \ln r + \theta)}$, where $r$ is the modulus of $z$ and $\theta$ is its argument. For $a < 0$, when $r < 1$, then $r^{a} > 1$ and similarly, for $a < 0$, when $r > 1$, then $r^{a} < 1$. 

To illustrate how the complex plane transforms in that case below we first plot the plane for the intervals $-2 \geq \operatorname{Re}(z) \leq 2$ and $-2 \geq \operatorname{Im}(z) \leq 2$ and show how that part of the plane transforms. The transformed plane seems to have a hole inside it. However, this is not the case. To see this, we extend the plane to include $-4 \geq \operatorname{Re}(z) \leq 4$ and $-4 \geq \operatorname{Im}(z) \leq 4$ as and plot this new transformation of the plane. We see that the hole in the middle is being filled in. To completely fill in the hole, we need to extend plane to r at infinity, which is obviously not possible in a numerical experiment such as this.

\begin{multicols}{2}
\begin{center}
\includegraphics[scale=0.43]{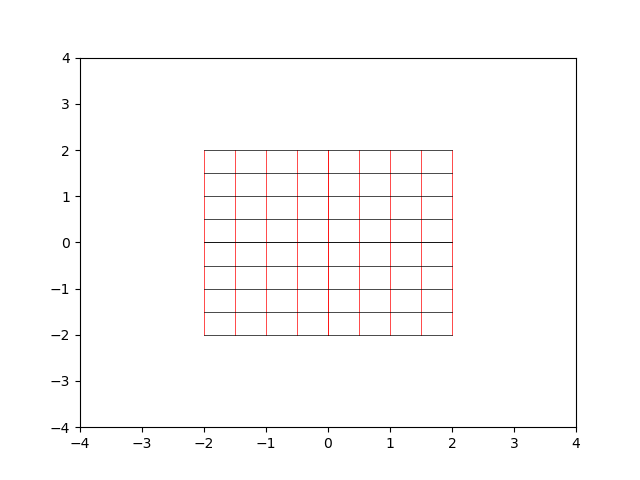}\\
Complex plane for $-2 \geq \operatorname{Re}(z) \leq 2$ and $-2 \geq \operatorname{Im}(z) \leq 2$.
\end{center}
\columnbreak
\begin{center}
\includegraphics[scale=0.43]{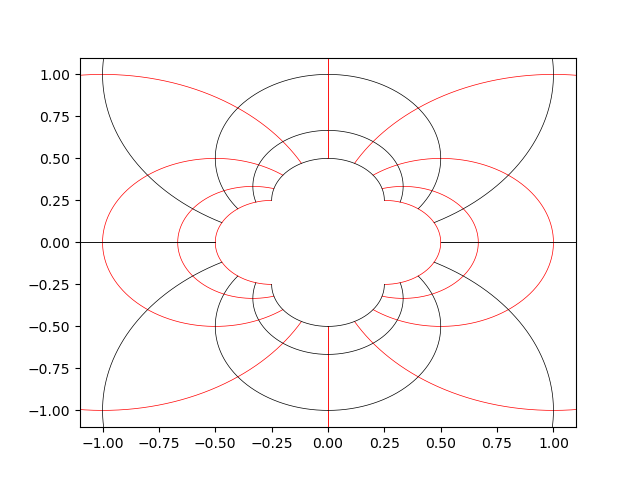}\\
$\alpha = -1$
\end{center}
\end{multicols}

\begin{multicols}{2}
\begin{center}
\includegraphics[scale=0.43]{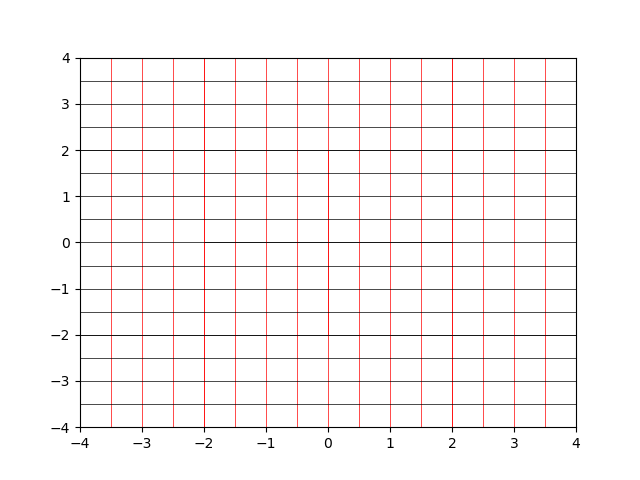}\\
Complex plane for $-4 \geq \operatorname{Re}(z) \leq 4$ and $-4 \geq \operatorname{Im}(z) \leq 4$
\end{center}
\columnbreak
\begin{center}
\includegraphics[scale=0.43]{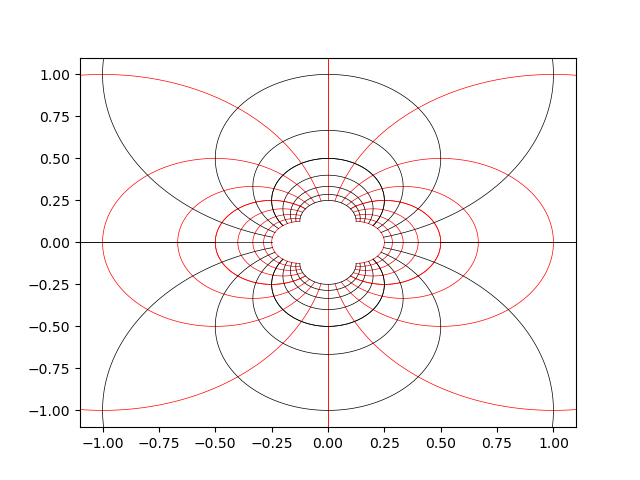}\\
$\alpha = -1$
\end{center}
\end{multicols}
\newpage 

\begin{multicols}{2}
\begin{center}
\includegraphics[scale=0.43]{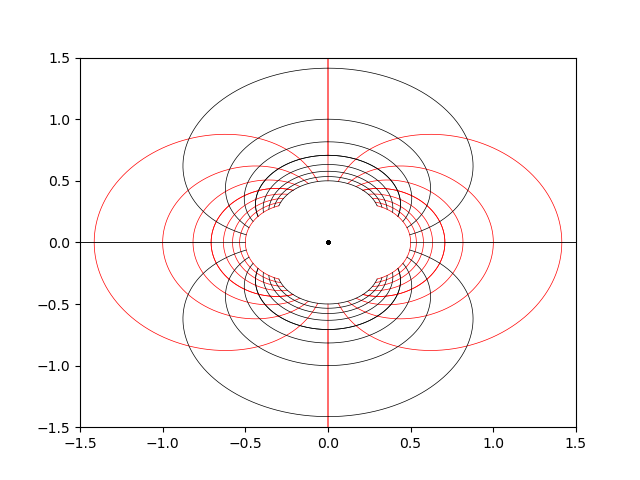}
$\alpha=-\frac{1}{2}$
\end{center}
\columnbreak
\begin{center}
\includegraphics[scale=0.43]{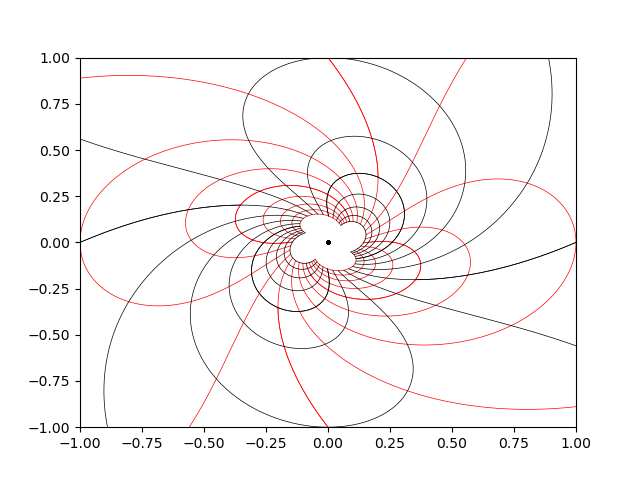}\\
$\alpha=-\frac{1}{2}\left(\frac{\sqrt{3}}{2} + \frac{1}{2}i\right)$
\end{center}
\end{multicols}

\begin{multicols}{2}
\begin{center}
\includegraphics[scale=0.43]{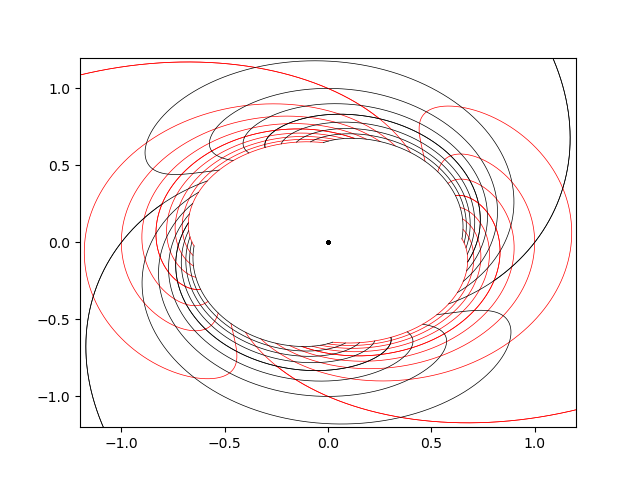}
$\alpha=-\frac{1}{2}\left(\frac{1}{2} + \frac{\sqrt{3}}{2}i\right)$
\end{center}
\columnbreak
\begin{center}
\includegraphics[scale=0.43]{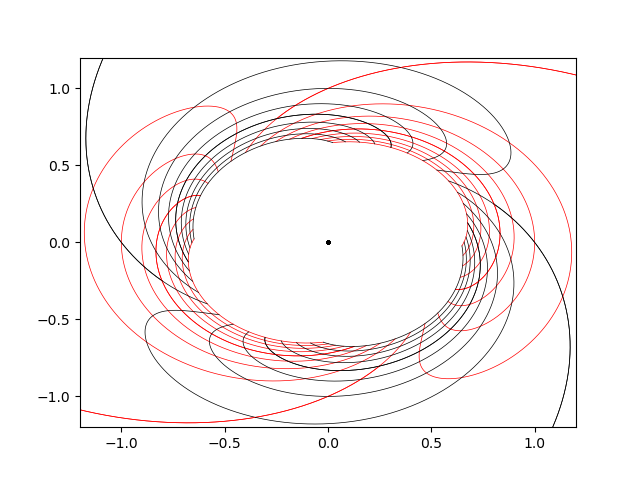}\\
$\alpha=-\frac{1}{2}\left(\frac{1}{2} - \frac{\sqrt{3}}{2}i\right)$
\end{center}
\end{multicols}

\begin{multicols}{2}
\begin{center}
\includegraphics[scale=0.43]{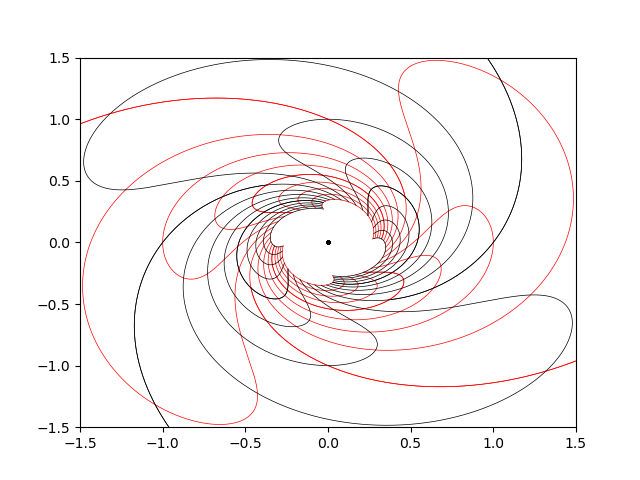}
$\alpha=-\frac{3}{2}\left(\frac{1}{2} + \frac{\sqrt{3}}{2}i\right)$
\end{center}
\columnbreak
\begin{center}
\includegraphics[scale=0.43]{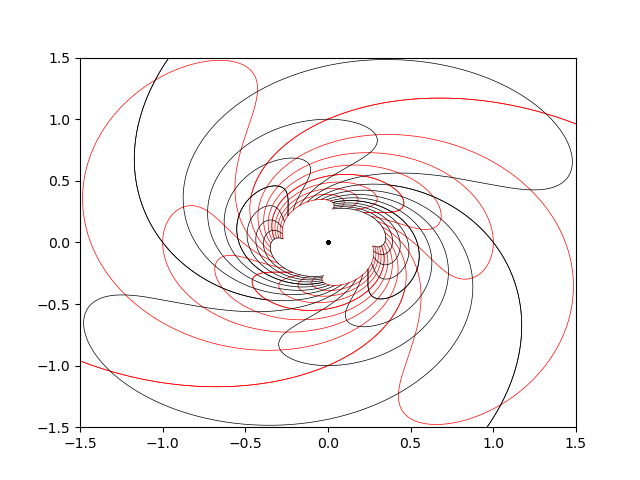}\\
$\alpha=-\frac{3}{2}\left(\frac{1}{2} - \frac{\sqrt{3}}{2}i\right)$
\end{center}
\end{multicols}

\mb{}

\end{document}